\documentclass{amsart}
\usepackage{amsmath,amssymb,amsthm,fullpage}
\usepackage{subfiles}
\usepackage[numbers]{natbib}

\newcommand{\Z}{{\mathbb Z}}
\newcommand{\N}{{\mathbb N}}
\newcommand{\R}{{\mathbb R}}

\newcommand{\ep}{{\varepsilon}}

\newcommand{\Vol}{\mathrm{Vol}}
\newcommand{\dist}{\mathrm{dist}}
\newcommand{\grad}{\nabla}

\newcommand{\diam}{\mathrm{diam}}
\newcommand{\vw}{\mathbf{w}}

\usepackage{braket}

\newcommand{\abs}[1]{\left\vert #1 \right\vert}

\newcommand{\parens}[1]{\left( #1 \right)}

\theoremstyle{plain}
\newtheorem{theorem}{Theorem}[section]

\newtheorem{corollary}[theorem]{Corollary}

\newtheorem{proposition}[theorem]{Proposition}
\newtheorem{lemma}[theorem]{Lemma}
\newtheorem{lem}[theorem]{Lemma}
\newtheorem{conjecture}{Conjecture}[section]

\newtheorem{remark}[theorem]{Remark}

\theoremstyle{definition}
\newtheorem{example}{Example}

\theoremstyle{definition}
\newtheorem{definition}{Definition}[section]

\usepackage{color}
\definecolor{red}{rgb}{.8,0,0}
\definecolor{green}{rgb}{0,.7,0}
\definecolor{blue}{rgb}{0,0,.8}

\title{Volume growth, curvature, and Buser-type inequalities in graphs}
\author{Brian Benson}
\address[Benson]{Department of Mathematics, University of California, Riverside, CA}
\email[Corresponding Author]{bbenson@ucr.edu}

\author{Peter Ralli}
\address[Ralli]{Program in Applied and Computational Mathematics, Princeton University, Princeton, NJ}
\email{pralli@princeton.edu}

\author{Prasad Tetali}
\address[Tetali]{School of Mathematics, Georgia Institute of Technology, Atlanta, GA}
\email{tetali@math.gatech.edu}

\date{}
\begin{document}
\maketitle

\begin{abstract}
We study the volume growth of metric balls as a function of the radius in discrete spaces, and focus on the relationship between volume growth and discrete curvature. We improve volume growth bounds under a lower bound on the so-called Ollivier curvature, and discuss similar results under other types of discrete Ricci curvature.

Following recent work in the continuous setting of Riemannian manifolds (by the first author), we then bound the eigenvalues of the Laplacian of a graph under bounds on the volume growth. In particular, $\lambda_2$ of the graph can be bounded using a weighted discrete Hardy inequality and the higher eigenvalues of the graph can be bounded by the eigenvalues of a tridiagonal matrix times a multiplicative factor, both of which only depend on the volume growth of the graph. As a direct application, we relate the eigenvalues to the Cheeger isoperimetric constant. Using these methods, we describe classes of graphs for which the Cheeger inequality is tight on the second eigenvalue. We also describe a method for proving Buser's Inequality in graphs, particularly under a lower bound assumption on curvature.
\end{abstract}

\section{Introduction}
\subsection{History and Motivation}
In Riemannian geometry there is a large and celebrated body of literature relating the Ricci curvature to various properties of the manifold, such as the Laplacian operator, the volume, the diameter, and various isoperimetric properties \cite{CY75,Buser82,Bishop}.  There has been much work in graphs and Markov chains studying the analogues of concepts that arise in Riemannian geometry, for example the Laplacian, isoperimetric constant and Cheeger inequalities \cite{Alon86,AC,Ch97}.  These successes have motivated the problem of defining the discrete Ricci curvature.  There have so far been several proposed definitions of discrete Ricci curvature \cite{S06,LV09,OV00,BE85,Oll07,EM12,CY16,Bauer2013,BS09}.  It is generally unclear whether or not any of these notions of curvature are equivalent, and in some instances examples illustrate that they are not equivalent.

It is preferable that a notion of discrete Ricci curvature would allow for similar results to those that hold for manifolds, such as relating global isoperimetric properties to the discrete curvature.  We should also hope that it is relatively easy to compute the discrete curvature.  In Riemannian geometry there are many results under the hypothesis of positive (or non-negative) curvature; if we can find similar results for graphs, we would like there to be large classes of interesting graphs that \emph{have} positive (or non-negative) curvature, and be able to make use of it in refining or strengthening various geometric and functional inequalities. 

As mentioned above, there have been many distinct definitions of the discrete Ricci curvature, each developed by taking a well-understood property of Ricci curvature in Riemannian manifolds and adapting it to the setting of graphs and Markov chains. In this work we will mainly focus on the Ollivier curvature, which is defined by the solutions to minimum transport problems between balls of small radius. The so-called Ollivier curvature was defined and developed significantly by Ollivier (although it was introduced earlier, independently by Sammer) \cite{Oll07,Sam05}.

To motivate this definition, we first briefly discuss the relationship between optimal transport and curvature in manifolds. Let $M$ be a Riemannian manifold with points $x,y$ which are close enough to be connected via a unique distance minimizing geodesic $\gamma$ and let $v$ be a direction at $x.$  We denote by the direction $w$ at $y,$ the parallel transport of $v$ along $\gamma$ to the point $y$ using the manifold's connection. Now consider $B(x,r)$ and $B(y,r),$ the metric ball of small radius $r>0$ centered at $x$ and $y$ respectively. We can move $B(x,r)$ along a small distance $\alpha>0$ in the direction $v$ by moving each $z \in B(x,r)$ in the following way: transport $v$ from $x$ to $z$ along the distance minimizing geodesic from $x$ to $z,$ call this direction $v_z.$ Then move a distance $\alpha$ from $z$ in the direction $v_z,$ corresponding to a point $z'$ in the manifold. We can use the same procedure with the vector $w$ at $y$ to move each point in $B(y,r)$ distance $\alpha$ in the direction of $w.$ If the Ricci curvature is positive, then the average of the distances between points in $B(x,r)$ and $B(y,r)$ will be further than their counterparts under the parallel transport of these metric balls. One the other hand, if the curvature is negative, on average, the distances between points in $B(x,r)$ and $B(y,r)$ will be closer than their counterparts under the parallel transport. Ollivier observed that the average distance can be replaced by the $L_1$-Wasserstein distance between uniform distributions on $B(x,r)$ and $B(y,r)$, and this metric is used in definition of the so-called Ollivier curvature, which can be used to recover the manifold's Ricci curvature (up to a factor) \cite{Oll07}.

Ollivier used this concept to help define the discrete Ricci curvature \cite{Oll07}. The metric balls $B(x,r)$ and $B(y,r)$ can also be defined on a graph where $r$ is a non-negative integer and $x$ and $y$ are vertices of the graph. Then the $L_1$-Wasserstein distance between the balls $B(x,r)$ and $B(y,r)$ determines a notion of curvature on the graph. 

While definitions of Ollivier curvature can be applied to any metric measure space, arguably its most fruitful use has been to define curvature in graphs with the graph distance and counting measure, for example \cite{BaJoLi11,BCLMP17,JoLi11}. That will also be our focus in this work:
A well-known fact due to Bishop is that a Riemannian manifold with a lower bound on its Ricci curvature will have the volume growth of its metric balls controlled by this lower bound \cite{Bishop}. Under many notions of discrete curvature it is unclear whether such a volume growth bound exists.  In this work we will present a volume growth that is interesting for regular graphs with a negative lower bound on Ollivier curvature.

We will also briefly discuss the $CDE'$ curvature, which was created by Bauer, Jost, and Liu \cite{Bauer2013}.  The $CDE'$ inequality is a modification of the $CD$ inequality of Bakry-\'{E}mery, which is a discrete generalization of the Bochner formula from Riemannian geometry.  Those authors demonstrated a version of the Li-Yau gradient estimate for graphs under the $CDE'$ curvature.  This is a result that does not have any known analogue in the setting of Ollivier curvature.

Volume growth estimates for Riemannian manifolds can also be applied to study the eigenvalues of the Laplace-Beltrami operator, denoted $\Delta_g,$ on the manifold. In fact, the relationship between the dimension, Ricci curvature, Cheeger constant, and spectrum of the Laplace-Beltrami operator on a closed Riemannian $n$-manifold has been well established. To remain consistent with the notation of the Laplace eigenvalues on graphs, denote by $\lambda_2(M)$ the first nonzero eigenvalue of $\Delta_g$ on $M.$\footnote{In geometry, the convention is to index the least positive eigenvalue by 1. However, we adopt the convention used in graph theory throughout.} Cheeger first showed that $\lambda_2(M) \geq h^2(M)/4,$ independent of the curvature or volume growth of the manifold \cite{Ch70}. Buser then proved that if the Ricci curvature of $M$ is bounded below by $-(n-1)\delta^2$ with $\delta \geq 0,$ then 
	\begin{equation*}
		\lambda_2(M) \leq 2\delta(n-1)h(M)+10h^2(M).
	\end{equation*}
Buser's original proof of this inequality used work relating the volume growth to the lower bound on the Ricci curvature due to Bishop \cite{Bishop} and Heintze and Karcher \cite{HK}. More recently, Agol proved a quantitative improvement of the estimate \cite{Agol}. Soon after, the first author proved an analogue giving upper bounds on every eigenvalue of $\Delta_g$ using only the dimension, a lower bound on Ricci curvature, and the Cheeger constant; the same quantities used in Buser's original inequality \cite{B15}. In each of these results, the lower bound on the Ricci curvature is necessary to control the volume growth of the level sets of the distance functions from the optimal Cheeger splitting. Further details are discussed in greater detail in Section \ref{sect:vg_lambda}. It should also be noted here that in \cite{Ledoux94}, Ledoux provided a simpler analytic proof of Buser's original result, and also followed up with a remarkable {\em dimension-free} improvement (see Theorem~5.2 in \cite{Ledoux04spectralgap}), where the constants are independent of the dimension of the manifold.

A problem of particular interest for graphs is the relationship between the isoperimetric constants and the spectral gap ($\lambda_2$) of the Laplacian of the graph. The Cheeger and Buser inequalities have analogues for graphs. Such relationships are frequently referred to in the literature as Cheeger-type inequalities, and relates the algebraic and geometric expansion properties of the graph. For the isoperimetric constant $h_{out}$, defined using the outer vertex boundary (also known as {\it vertex expansion}, and reviewed in the next section), the Cheeger inequalities \cite{BHT} are

\begin{align}\frac{\left (\sqrt{1+h_{out}}-1\right )^2}{2d}\leq \lambda_2 \leq h_{out}\end{align} for any $d$-regular graph.
A long-standing problem of general interest is to determine the class of graphs for which the lower inequality $\lambda_2 \approx h_{out}^2$ is tight.

There is a previous proof of a discrete Buser's inequality, which states that under the condition of non-negative Ricci curvature (in the sense of the $CD$ inequality of Bakry-\'{E}mery), the lower Cheeger inequality is tight \cite{KKRT15}.  The proof method relies on decomposing a candidate Cheeger-optimizing vertex set as a linear sum of eigenfunctions of the Laplacian, and analyzing the behavior of those functions under the heat flow operator $P_t$, which can be seen as the evolution of the random walk on the graph.  This proof was recently extended to bound the higher eigenvalues of the Laplacian \cite{LMP}.

\subsection{Summary of Results}
We prove specific results bounding the spectrum using only volume growth. To summarize, let $A$ be a subset of the vertex set of a graph $G.$ In Theorem \ref{theo:eigen2}, we prove that $\lambda_2(G)$ can be bounded from above by a weighted discrete Hardy inequality which depends only on bounds on the volume growth of $A.$ Such Hardy inequalities are well understood and we combine our work with results of Miclo \cite{Miclo} to give quantitative estimates on the first eigenvalues in terms of volume growth and $h_{out}(G),$ which are stated in Theorem \ref{thm:lambda_2}. We also prove in Theorem \ref{theo:eigenhigher} that higher eigenvalues $\lambda_k(G)$ where $k\geq 2$ can be bounded above by the eigenvalues of matrices which depends only on volume growth bounds. 

As an application of the relationship between the spectrum and volume growth, we suggest an alternate proof method of Buser's inequality on graphs, which instead uses a bound on volume growth around a set achieving the optimal Cheeger constant. Such approach is inspired by the original proof of Buser \cite{Buser82}, in the continuous setting of manifolds, as well as subsequent improvements by Agol \cite{Agol} and the first author \cite{B15}. In particular, we can extend the proof of our Buser-type inequality on graphs to bound the higher eigenvalues of the Laplacian.  A similar result was demonstrated for manifolds in previous work of the first author.

It is interesting to note that a bound on discrete curvature is only used in our methods to find a suitable volume growth function.  If a bound on volume growth for a specific graph (or a family of graphs) exists under some other condition unrelated to curvature, our theorems immediately admit upper bounds on eigenvalues. In particular, we prove that any graph whose ``shells'' -- sets of vertices a fixed distance from a (Cheeger-optimal) isoperimetric cut-set -- have volume bounded from above by the volume of the cut-set satisfies 
\begin{align*}\lambda_2 \leq \frac{27}{2} h_{out}^2.\end{align*}
Therefore, the lower Cheeger inequality is tight up to a multiplicative factor $c = c(d)$ depending only on degree $d$ of a $d$-regular graph. This result appears in Example \ref{ex:constant}. In Example \ref{ex:HigherEigen}, we show that when the volume growth is bounded by a constant, that higher eigenvalues can be bounded by higher Cheeger constants. Specifically, the higher Cheeger constant $h_{out}(n)$ (arising from splitting the graph into $n$ subgraphs). Specifically, under the same aforementioned volume growth assumptions, we have, for any positive integers $n$ and $k$,
	\begin{equation*}
    \lambda_k \leq k^2\left (\frac{27\pi^2}{16}+o(1)\right )h_{out}(n)^2.
	\end{equation*}

\subsection*{Acknowledgments} 
 The authors are indebted to the anonymous referees for a very careful reading of the original manuscript. Their many corrections and suggestions have helped remove errors and improved the presentation of the results. The first author would like to thank the School of Mathematics at the Georgia Institute of Technology for their support and hospitality during several visits to work on this project. The second and third authors were supported in part by the NSF grants DMS-1407657 and DMS-1811935.  The second author was also supported in part by the AMO grant W911NF-14-1-0094.  A portion of the research was conducted while the first author was a visiting assistant professor at Kansas State University and the second was a graduate student at Georgia Institute of Technology.

\section{Notation}

A graph $G = (V,E)$ has a vertex set $V$ and an edge set $E$ that contains $2$-element subsets of $V$.  A finite graph is one where $V$ is a finite set. 
If $\{x,y\}\in E$, we say that $x$ and $y$ are neighbors, denoted $x\sim y$.  A common shorthand is that the edge $\{x,y\}$ may be denoted $xy$.  The degree of a vertex $x$ is the number of neighbors of $x$.  A locally finite graph is one where each vertex has a finite set of neighbors.  For some integer $d>0$, a $d$-regular graph is one where each vertex has exactly $d$ neighbors.  Clearly such a graph is also locally finite.

A walk on $G$ is a series of vertices $v_0,v_1,\dots, v_n$ so that $v_{i-1}v_i$ is an edge for all $i = 1,\dots n$.  A graph is connected if every pair of vertices comprises the two ends of some walk.  For the rest of this work, we will only consider connected graphs.

Let $G$ be a $d$-regular graph.  The adjacency operator $A$ on the space $\{f:V\to\R\}$ is defined by the equation \begin{align*}Af(x) = \tfrac{1}{d}\sum_{y:x\sim y}f(y),\end{align*} and the Laplacian operator $\Delta$ on the same space is \begin{align*}\Delta f(x) = \tfrac{1}{d}\sum_{y:y\sim x}f(x) -f(y).\end{align*}  In other words, one has $\Delta = I- A,$ where $I$ is the identity operator satisfying $If =f.$ (In other parts of the literature, these operators are sometimes referred to as the \emph{normalized} adjacency operator and \emph{normalized} Laplacian.)

Observe that for a finite graph $\Delta$ is a symmetric and positive semi-definite matrix; as such, the eigenvalues of $\Delta$ are all real and non-negative. By convention we write the eigenvalues of $\Delta$ (counting multiplicities) as $\lambda_1(\Delta),\lambda_2(\Delta),\dots$ with $\lambda_1(\Delta)\leq \lambda_2(\Delta) \leq \dots$.  It is well-known that $\lambda_1(\Delta)$ is achieved by the eigenfunction $f\equiv 1$ with $\lambda_1 = 0$.  The spectral gap of $G$ is the difference between the two least eigenvalues of $\Delta,$ which is $\lambda_2(\Delta)$ since $\lambda_1(\Delta)=0.$  Often we write these values as $\lambda_1(G),\lambda_2(G),\dots,$ even suppressing the graph $G$ when clear.

Let $G$ be a $d$-regular, finite graph.  
 For a vertex subset $A\subset V$, define the edge boundary $\partial A$ to be $\set{\{x,y\}\in E: x\in A; y\notin A}.$
The (Cheeger) edge isoperimetric constant is defined as
$ h(G) = \min_{A}\frac{|\partial A|}{d|A|}\,,$ where the minimization is over all sets $A$ with $0<|A|\leq \tfrac{|V|}{2}$.

Cheeger-type inequalities 
relate edge and vertex isoperimetric constants to the spectral gap of the Laplacian of the graph. 
In particular, classical results (e.g., \cite{AM, AC,Tan84}, to cite just a few) show that $\frac{h^2}{2}\leq \lambda_2 \leq 2h\,.$

\smallskip
In addition to the edge boundary of a set $A\subset V$, one can define two different vertex boundaries:
The inner vertex boundary is $\partial_{in}A = \{x\in A:\exists y\sim x; y\notin A\}$, and the
outer vertex boundary is $\partial_{out}A = \{y\notin A:\exists x\sim y; x\in A\}$.

Following \cite{BHT}, one has the (Cheeger) vertex isoperimetric constants using the vertex boundaries:
\[
h_{in}(G) = \min \frac{|\partial_{in}A|}{|A|}\ \  \mbox{ and } \ \ 
h_{out}(G) = \min \frac{|\partial_{out}A|}{|A|}\,.
\]
In all cases, the minimization is over non-empty vertex sets with $|A|\leq \tfrac{1}{2}|V(G)|$.  
Observe the trivial bounds $h(G)\leq h_{in}(G)\leq d\cdot h(G)$ and $h(G)\leq h_{out}(G)\leq d\cdot h(G),$ so bounds on the vertex constants imply bounds on the edge constants and vice versa.

There are also a pair of Cheeger-type inequalities for each of these isoperimetric constants \cite{BHT,Alon86}; in particular, 
for the outer vertex boundary, the inequalities are:
$$\frac{\left (\sqrt{1+h_{out}}-1\right )^2}{2d^2}\leq \lambda_2 \leq \frac{h_{out}}{d}\,,$$
where the additional factors of $d$ in the denominators as compared to the edge Cheeger inequalities arise from the need to normalize $h_{out}$.

We now define the Ollivier curvature, which relies on concepts of optimal or minimum transport. Let $X$ be a measurable metric space with metric $\dist(\cdot,\cdot)$, and let $\mu,\nu$ be two probability measures on $X$.  
The $L_1$ Wasserstein (also known as minimum-transport or earth-mover) distance \cite{AGS11a} is 
\begin{align*}W_1(\mu,\nu) = \inf_m \int_{X\times X} \dist(x,y)\, m(x,y),
\end{align*} 
where the minimum is taken over all joint distributions (couplings) $m$ on $X\times X$ 
with left marginal $\mu$ and right marginal $\nu$. 
Qualitatively, we wish to transport the distribution $\mu$ to $\nu$.  
Here $m$ is a movement plan that moves probability mass $m(x,y)$ from $x$ to $y$, and we choose $m$ to minimize the average distance moved by the mass.

There is a well-known dual to the minimization problem \cite{AG11}:
\begin{align} 
W_1(\mu,\nu) = \sup_{f\in \text{Lip}(1)} \int_X f \nu - \int_X f\mu,\end{align}  
where $\text{Lip}(1)$ is the space of functions with Lipschitz constant equal to one. A maximizing function for this equation is sometimes known as a \emph{Kantorovich potential}.

Observe that if the probability measures $\mu_x$ and $\mu_y$ (on $X$) have finite support, both the primal and dual characterizations of $W_1(\mu_x,\mu_y)$ are linear programs on a finite set of variables.  All the probability distributions we will consider in our discussion of Ollivier curvature will be of this type.  For these distributions we will use the notation of finite sums indexed by vertices rather than integrals over the measure space of vertices.

Let $G$ be a locally finite connected graph and $x\in V(G)$ a vertex with degree $d_x$.  Define a probability measure $\mu_x$ on $V$ so that \begin{align*}\mu_x(v) = \begin{cases} 
      \tfrac{1}{2} & \text{if }v = x \\
      \frac{1}{2d_x} & \text{if }v \sim x \\
      0 & \text{otherwise.}
   \end{cases}\end{align*}  
Here, think of taking one step of a random walk starting at $x$ and with laziness parameter $1/2$.

\begin{definition} If $x,y\in V$, the Ollivier curvature with is \begin{align}\kappa(x,y) = 1-\frac{W_1(\mu_x,\mu_y)}{d(x,y)}.\end{align}
\end{definition}

The choice of laziness parameter is to some extent not important: suppose we vary the value $p = \mu_x(x) = \mu_y(y)$.  When $p\geq \max \parens{\frac{1}{d_x+1}, \frac{1}{d_y+1}},$ the optimal transport plans and the value $\kappa(x,y)$ vary linearly with $1-p$ \cite{BCLMP17}.

For later sections, we need some basic and well-known facts about Ollivier curvature which we now briefly review.
\begin{theorem}[Neighbors minimizing curvature (Y. Ollivier, \cite{Oll07})]\label{thm:oll_minim} Suppose that $\kappa(u,v)\geq k$ whenever $u,v\in V$ are neighboring vertices.  Then, also for any $x,y\in V$ (not necessarily neighbors), we have $\kappa(x,y)\geq k.$\end{theorem}  
In other words, it is equivalent to say that $k$ is a global lower bound on curvature and that $k$ is a lower bound on the curvature between each pair of neighbors. We give a quick proof due to Ollivier \cite{Oll07}. \begin{proof}Observe that if $u\sim v$, then $W_1(\mu_u,\mu_v) = 1 - \kappa(u,v) \leq 1 - k.$

Let $x = x_0, x_1, \ldots, x_l = y$ be a geodesic path in $G$.  Because $W_1$ is a metric, it follows that 		$$W_1(\mu_x,\mu_y) \leq \sum_{i=1}^l W_i(\mu_{x_{i-1}},\mu_{x_i}) \leq (1-k)d(x,y),$$
and $\kappa(x,y) \geq 1 - \frac{(1-k)d(x,y)}{d(x,y)} = k.$
\end{proof}

Ollivier also provided a result for estimating curvature on product graphs. Later, we will use the following result to apply our techniques to the discrete hypercube.  In our notation, the graph product $G\square H$ has vertex set $V(G)\times V(H)$, and edges $(x,y)\sim (w,z)$ if $(x,y)$ and $(w,z)$ are adjacent in one component and identical in other other.
\begin{theorem}[Ollivier curvature tensorization (Y. Ollivier, \cite{Oll07})]\label{thm:oll_tens}
Let $G$ be a $d$-regular graph, and denote 
$G\square G\square \cdots \square G$ 
with $r$ terms in the product by $G^r$.  Suppose that for every $x,y\in V(G)$, it holds that $\kappa(x,y)\geq k.$  Then for every $x',y'\in V(G^r),$ we have that $\kappa(x',y')\geq \frac{k}{r}.$
\end{theorem}

Again, we provide a short proof from Ollivier's original work \cite{Oll07}.
\begin{proof}
Let $x$ and $y$ be neighbors in $G^r$. By Theorem \ref{thm:oll_minim}, it suffices to show  $\kappa(x,y)>r$.  Without loss of generality we may assume $x = (x_1,x_2,\dots x_r)$ and $y = (y_1,x_2,\dots x_r)$.  Let $f_1$ be the Kantorovich potential satisfying \[\sum_v f_1(v) \mu_{y_1}(v) - \sum_v f_1(v) \mu_{x_1}(v) = 1-\kappa(x_1,y_1) \leq 1-k.\]
Define $f(z_1,\dots, z_r) = f_1(z_1)$, then we see that \begin{align*}&\sum_w f(w) \mu_y(w) - \sum_w f(w) \mu_x(w) \\ = &\frac{1}{r}\parens{\sum_v f_1(v) \mu_{y_1}(v) - \sum_v f_1(v) \mu_{x_1}(v)} + \frac{r-1}{r}\parens{f_1(y_1)-f_1(x_1)} \\ \leq &\frac{1}{r}(1-k) + \frac{r-1}{r} = 1 - \frac{k}{r}.\end{align*}  In other words, we have $\kappa(x,y)\geq \frac{k}{r}$.
\end{proof}
    
\section{Volume Growth and Spectral Gap in Manifolds}\label{sect:vg_lambda}
In this section we will outline the proof of Buser-type results on manifolds, particularly following the work of Buser \cite{Buser82}, of Agol \cite{Agol}, and of the first author \cite{B15}.  In the following sections we will develop analogous methods to bound the spectral gap and higher eigenvalues in graphs.

Let $M$ be an  $n$-dimensional manifold and let $A$ and $B$ be a Cheeger-minimizing partition of $M$, so that their common boundary $\Sigma = \partial A = \partial B$ satisfies \begin{align*}h(M) = \frac{\Vol(\Sigma)}{\min\parens{\Vol(A),\Vol(B)}}.\end{align*}

The minimax principle tells us that $\lambda_2(M)\leq \max\set{\lambda_1(A),\lambda_1(B)}$ where eigenfunctions $f$ of $A$ (similarly $B$) corresponding to eigenvalue $\mu$ satisfy the Dirichlet boundary conditions \begin{align*}\begin{cases} \Delta f = \mu f \text{ on }A \\ f(\Sigma) = 0. 
\end{cases}\end{align*} 
Here, $\lambda_1(A)$ (similarly $\lambda_1(B)$) is the least non-zero value $\mu$ for which an eigenfunction exists. Without loss of generality assume that $\lambda_1(A)\geq \lambda_1(B)$.
The Rayleigh principle tells us that $\lambda_1(A)$ of a manifold is achieved by minimizing the Rayleigh quotient $\int_A\abs{\nabla f}^2/\int_A f^2$ over functions satisfying the boundary condition $f(\Sigma) = 0$. For more details, see \cite{HJ, Chavel, Lablee}.

Buser's idea is to use a test-function for this Rayleigh quotient that depends linearly on the distance from $\Sigma,$ which we denote $\dist_{\Sigma}(p).$ Specifically, one constructs 
\begin{align*}f(p) = \begin{cases} \dist_{\Sigma}(p) &\text{ if } \dist_{\Sigma}(p) \leq t\\ t &\text{ if } \dist_{\Sigma}(p) \geq t.\end{cases}\end{align*}

Define $A(t) = \{p\in A: \dist_{\Sigma}(p)\leq t\}$.  Buser observes that \begin{align*}&\int_A \abs{\nabla f}^2 \leq \int_{A(t)}1 = \Vol(A(t))\text{ and}\\ &\int_A f^2 \geq \int_{A-A(t)}t^2 = t^2\parens{\Vol(A)-\Vol(A(t)}. 
\end{align*}
Now, for any $t> 0$ satisfying $\Vol(A)> \Vol(A(t))$, one sees that
	\begin{align*}
		\lambda_2(M)\leq \frac{\Vol(A(t))}{t^2\parens{\Vol(A)-\Vol(A(t)}}\,.
    \end{align*}

What remains is to bound $\Vol(A(t))$.  In this step, Buser uses a global lower bound on Ricci curvature and the crucial assumption that $\Sigma$ is a Cheeger-optimal cut-set.
Suppose $N$ is a compact hypersurface (codimension-1 submanifold) of $M$. Further, assume that the planes of $M$ containing a tangent vector of a geodesic segment which minimizes the distance to $N$ have sectional curvatures are bounded below by $\delta$. A consequence of the Heintze-Karcher comparison theorem \cite{HK} is the following volume growth bound:
There exists $\nu_{\delta} \in C^{\infty}[0,\infty)$ such that for all $\tau\geq 0$, we have
		$$\Vol_{n-1} \big ( \dist^{-1}(\tau)\big )\leq \Vol_{n-1}(\Sigma)\nu_{\delta}(\tau).$$

Now, the volume growth bound $\Vol(A(t)) \leq \int_0^t \nu(s)\Vol(\Sigma) \, ds$ can be applied (when clear we will suppress the $\delta$ in $\nu_{\delta}$). Specifically, Buser finds the bound 
\begin{align*}\lambda_2(M) &\leq\frac{\int_0^t \nu(s)\Vol(\Sigma)ds}{t^2 \Vol(A) - t^2\int_0^t \nu(s)\Vol(\Sigma)ds}\\ &\leq \frac{h\int_0^t\nu(s)ds}{t^2(1-h\int_0^t\nu(s)ds)}
\end{align*} 
for $\lambda_2(M)$ in terms of the curvature (again, because the volume growth function $\nu$ depends on curvature), the Cheeger cut-set $A$ and boundary $\Sigma$. We will not reproduce the remainder of Buser's proof \cite{Buser82}, which is somewhat technical, except to state the result: \begin{theorem}[Buser's Inequality, (P. Buser 1982)] If $M$ is an $n$-dimensional manifold with $-(\delta^2)(n-1)$ as a lower bound on curvature (for some $\delta\geq 0$), then 
	\begin{align*}\lambda_2(M)\leq c(\delta h(M) + h(M)^2),
	\end{align*}
where $c$ is a universal constant.\end{theorem}

More recently, Agol observed that the constant in Buser's proof can be improved by optimizing over \emph{all} possible test-functions that depend on the distance from $\Sigma$, not just those that grow linearly up to some critical distance $t$ \cite{Agol}. While reformulating Agol's result using Sturm-Liouville theory, the first author showed that the method can be extended to give bounds on the higher eigenvalues \cite{B15}. One begins with the observation that 
	\begin{equation}\label{eq:Courant}
    	\lambda_{2k}(M)\leq \max\parens{\lambda_k(A),\lambda_k(B)},
	\end{equation}
where $\lambda_1(M),\lambda_2(M),\dots$ are the eigenvalues of $M$ in increasing order and $A$ and $B$ have the properties that 
	\begin{itemize}
    	\item $B=A^{\complement},$
        \item $\Vol(A) \leq \Vol(B),$
    	\item $A \cap B=\partial A=\partial B =: \Sigma,$
    	\item $h(M)=\frac{\Vol(\Sigma)}{\Vol(A)}.$
    \end{itemize}
We denote $D$ to be the set $A$ or $B$ that achieves the maximum in Equation \ref{eq:Courant}.  Here, the Rayleigh quotient is $$\lambda_k(A)
			=\inf_U \sup_{f \in U} \frac{\int_D \|\grad(f)\|^2 \, d\Vol}{\int_D f^2 \, d\Vol},$$
		where $U$ is the set of $k$-dimensional subspaces of the Sobolev space $H_0^1(D)$ on which $f(\Sigma) = 0$. Limiting to only those functions $f$ that depend on the distance from $\Sigma$, the co-area formula implies that 
\begin{align}\label{eqn:rayleigh2}\lambda_k (D) \leq \inf_V \sup_{f \in V} \frac{\int_0^{\infty}f'(s)^2\Vol\big (\dist_{\Sigma}^{-1}(s) \big )\, ds}{\int_0^{\infty} f(s)^2 \Vol\big (\dist_{\Sigma}^{-1}(s)\big )\, ds},\end{align} where $\dist^{-1}_{\Sigma}(s)$ is the set $\{p\in D: \dist(p,\Sigma) = s\}$ and $V$ is the set of all $k$-dimensional subspaces of $H_0^1[0,\infty)$ on which $f(0) = 0$.

Given a Heintze-Karcher-type growth bound $\Vol(\dist_{\Sigma}^{-1}(s))\leq \nu(s)\Vol(\Sigma)$, observe that 
\begin{align*}
\Vol(A(s)) = \int_0^s \Vol\left (\dist_{\Sigma}^{-1}(\tau)\right )\, d\tau \leq \Vol(\Sigma)\int_0^s \nu(\tau)d\tau.
\end{align*}

Because $\Sigma$ is the Cheeger-achieving boundary and $\dist_{\Sigma}^{-1}(s)$ is the boundary for some other non-Cheeger-achieving partition of $M$, we have 
\begin{equation*}h(M) = \frac{\Vol(\Sigma)}{\min\parens{\Vol(A),\Vol(B)}}\end{equation*}
and also  
	\begin{equation*}
			h(M)\leq 
            \frac{\Vol(\dist_{\Sigma}^{-1}(s))}{\min \parens{\Vol(A\setminus A(s)),\Vol(B\cup A(s))}}.
    \end{equation*}

In the case that $\Vol(B\cup A(s))\geq \Vol(A\setminus A(s))$, we have 
\begin{equation*}
h(M)\leq 
\frac{\Vol(\dist_{\Sigma}^{-1}(s))}{\Vol(A\setminus A(s))},
\end{equation*}
and so
\begin{align*}
\Vol(\dist_{\Sigma}^{-1}(s))&\geq h(M)\Vol(A) - h(M)\Vol(A(s)) \\ 
&\geq \Vol(\Sigma) - h(M)\Vol(\Sigma)\int_0^s\nu(\tau)d\tau\\ 
&=\Vol(\Sigma)\parens{1-h(M)\int_0^s\nu(\tau)d\tau}.
\end{align*}

In the other case, we have that $\Vol(B\cup A(s))\leq \Vol(A-A(s))$.  As such, it follows that $\Vol(B)\leq \Vol(A),$ in other words, the set $B$ is the Cheeger minimizing set.  We find that
\begin{align*}\Vol(\dist_{\Sigma}^{-1}(s))\geq h(M)\Vol(B\cup A(s))\geq h(M)\Vol(B) = \Vol(\Sigma),\end{align*} with the last equality following from the definition of $\Sigma$.

Combining both cases, the first author achieves the lower bound 
\begin{align*}\Vol(\dist_{\Sigma}^{-1}(s))\geq \Vol(\Sigma)\parens{1-h(M)\int_0^s\nu(\tau)d\tau}.\end{align*}

Observe that this bound is only meaningful for values of $s$ where 
\begin{align*}h(M)\int_0^s \nu(\tau)d\tau\leq 1.\end{align*}
Because the parameter $s$ is continuous, one can always apply such values. This is one of several ways in which the discrete formulation on graphs presents a challenge which does not appear in the related continuous result on Riemannian manifolds.

Define now $T$ to be the value for which $h(M)\int_0^T \nu(\tau)d\tau = 1.$
With both an upper and lower bound for $\Vol(\dist_{\Sigma}^{-1}(s))$, it is possible to plug those bounds into Equation \ref{eqn:rayleigh2}, truncating the integrals at $T$, to obtain the bound 

\begin{align}\label{eqn:rayleigh3}\lambda_k(A) \leq \inf_{W} \sup_{f\in W} 
\frac{\int_0^T f'(s)^2\nu(s)\, ds}
{\int_0^T f(s)^2\left (1-h\int_0^{s} \nu(\tau)\, d\tau \right )\, ds},\end{align} where $W$ is the set of $k$-dimensional subspaces of $H_0^1[0,T]$ in which $f(0) = 0$.

What remains is the technical problem of finding the function $f$ that minimizes the Rayleigh quotient in Equation \ref{eqn:rayleigh3}. Such a function is an eigenfunction of a Sturm-Liouville problem, which leads to an eigenvalue comparison. Specifically, the eigenvalues of the Laplacian on the manifold are bounded above by the eigenvalues of a Sturm-Liouville (ODE) problem which depends only on the same data as in Buser's original inequality; namely the Cheeger constant, dimension, and Ricci curvature lower bound. For more details, see \cite{B15}.

We will see in Section \ref{sect:Spectrum} that in the discrete case, higher eigenvalues of the graph $\lambda_k(G)$ can be bounded by the eigenvalues of a tridiagonal matrix times a multiplicative factor. The entries of the matrix only depend on the bounds on volume growth, which can be given in terms of several notions of the graph's curvature. Further, the multiplicative factor can be interpreted using the upper bound on the volume growth of the graph and the outer vertex Cheeger constant or its analogues corresponding to splitting the graph into more than two subgraphs.

\section{The Relationship Between Volume Growth and Curvature}\label{sect:vg_cde}

In Section \ref{sect:Spectrum}, we develop the relationship between the spectrum of the Laplace operator on a graph and the volume growth of subsets of the graph. Our goal is to also develop the connection between the spectrum, notions of curvature, and the Cheeger constant of the graph in the form of Buser-type inequalities. To allow us to make these connections in Section \ref{sect:Spectrum}, in this section we discuss volume growth in graphs under several notions of a curvature lower bound.

\subsection{Bounds under $CDE'$ curvature}

We first present a volume growth bound under the so-called $CDE'$ inequality.  This notion of discrete curvature is a variant of the $CD$ inequality, which was introduced by Bakry-\'Emery in \cite{BE85}.  The $CDE'$ inequality was introduced by Bauer et al.\cite{Bauer2013}.  We present only the definition of $CDE'(K,N)$ herein, for a full discussion the reader can consult the paper of Bauer et al.

Let $f,g:V(G)\to \R$ be  functions and $x\in V(G).$ The field-squared operator $\Gamma(f,g)(x)$ is defined by \begin{align*}
\Gamma(f,g)(x) = \frac{1}{2d_x}\sum_{y:y\sim x}\parens{f(x)-f(y)}\parens{g(x)-g(y)}.
\end{align*}
A graph $G$ is said to satisfy the $CDE'(K,N)$ inequality at $x$ if for every function $f:V(G)\to \R^+$, \begin{align*}f^2 \parens{\tfrac{1}{2}\Delta \Gamma(\log f,\log f) - \Gamma(\log f, \Delta \log f)}\geq \frac{1}{N}\parens{\Delta \log f}^2 + K\Gamma(f,f).\end{align*}
In this case, we say that $K$ is a lower bound on the $CDE'$ curvature of $G$ at $x$ with dimension $N$.

In a follow-up work \cite{HLLY15}, a volume growth bound was discovered under a lower bound on $CDE'$ curvature:

\begin{theorem}(Horn, Lin, Liu \& Yau \cite{HLLY15})\label{thm:vg_HLLY} Let $G$ be a locally finite graph satisfying $CDE'(n,0)$. Then there exists a constant $C$ depending on $n$ such that for all $x \in V$ and any integers $r,s$ with $r \geq s$:
	\begin{equation}\label{eq:HLLYVol}
		|\dist^{-1}_x(r)| \leq C \left ( \frac{r}{s}\right )^{\frac{\log(C)}{\log(2)}}|\dist^{-1}_x(s)|.
	\end{equation}
\end{theorem}

Note that a similar volume growth bound (albeit with a multiplicative factor $\sqrt{d}$ in the base of the exponent) is implicit in the recent article \cite{Munch19} under the $CD(n,0)$ inequality, and a similar corollary can be obtained by the following method.  We use this bound on ball volumes to prove the following bound on shell volumes:
\begin{corollary}Let $G$ be a graph satisfying $CDE'(n,0)$ at all vertices $x\in V(G)$.  Let $\Sigma \subset V$, and let $C = C(G)$ be the constant from Theorem \ref{thm:vg_HLLY}, let $r>0$.  Then 
$$|\dist^{-1}_{\Sigma}(r)| \leq d|\Sigma|  C(r-1)^{\frac{\log(C)}{\log (2)}}
	\left [C \left ( \frac{r}{r-1}\right )^{\frac{\log(C)}{\log(2)}}-1 \right ].$$
\end{corollary}
\begin{proof}Letting $s=r-1$ in Equation \ref{eq:HLLYVol}, the estimate becomes
	$$|\dist^{-1}_x(r)| \leq C \left ( \frac{r}{r-1} \right )^{\frac{\log(C)}{\log(2)}}|\dist^{-1}_x(r-1)|.$$
Since we are interested in counting vertices with distance exactly $r$ from $x$, we wish to consider $|\dist^{-1}_x(r)|-|\dist^{-1}_x(r-1)|$, so we subtract $|\dist^{-1}_x(r-1)|$ from both sides of the previous inequality to give
	$$|\dist^{-1}_x(r)|-|\dist^{-1}_x(r-1)| \leq \left [ C \left (\frac{r}{r-1} \right )^{\frac{\log(C)}{\log(2)}} -1 \right]|\dist^{-1}_x(r-1)|.$$
In fact, we want to consider the set of vertices with distance exactly $r$ from $\Sigma$.  We can sum over all $x \in \Sigma$ on both sides of the previous equation to give
	$$|\dist^{-1}_{\Sigma}(r)|\leq \sum_{x \in \Sigma} |\dist^{-1}_x(r)|-|\dist^{-1}_x(r-1)| \leq 
	\sum_{x \in \Sigma} \left [ C \left (\frac{r}{r-1} \right )^{\frac{\log(C)}{\log(2)}} -1 \right]|\dist^{-1}_x(r-1)|.$$
Simplifying, we find
	$$|\dist^{-1}_{\Sigma}(r)| \leq |\Sigma| \left [C \left ( \frac{r}{r-1}\right )^{\frac{\log(C)}{\log(2)}}-1 \right ] 
	\max_{x \in \Sigma} |\dist^{-1}_{\Sigma}(r-1)|.$$
Now we wish to estimate the term $\max_{x \in \Sigma} |\dist^{-1}_{\Sigma}(r-1)|$ and we will again apply Equation \ref{eq:HLLYVol}, this time  we replace $r$ with $r-1$ in the formula and take $s=1$. As a result, our estimate becomes 
	$$|\dist^{-1}_{\Sigma}(r)| \leq d|\Sigma|  C(r-1)^{\frac{\log(C)}{\log (2)}}
	\left [C \left ( \frac{r}{r-1}\right )^{\frac{\log(C)}{\log(2)}}-1 \right ],$$
observing that $|\dist^{-1}_x(1)| = d$.\end{proof}

\subsection{Bounds under Ollivier curvature}

Next, we will find upper bounds on the shell volume $|\dist^{-1}_x(i)|$ in terms of a lower bound on Ollivier curvature.  It is simple to convert such bounds into bounds on the ball volume (analogous to the Bishop-Gromov Volume Comparison Theorem \cite{Bishop}) with the equation $$|B_x(r)| = \displaystyle \sum_{i=0}^r |\dist^{-1}_x(i)|.$$

In this area, there is a previous result due to Paeng \cite{Paeng12}.

\begin{theorem}[Paeng \cite{Paeng12}] Let $G$ be a graph with maximum degree $D$.  Let $r$ be an integer with $0\leq r\leq \diam(G)$.  Assume that $\kappa(x,y) \geq k$ for all $x,y\in V$. 
	\begin{align}|d^{-1}_x(r)| \leq D^r \prod_{m=0}^{r-1} \left ( 1-\frac{k}{2}m \right )\,.\end{align}
\end{theorem} 

These bounds are only useful in the case that $k > 0$: if we set $k=0$ above, we see only the trivial result that $|f^{-1}(r)|\leq D^r$.
In the case $k > 0$, we see that $|\dist_x^{-1}(\lceil 2/k + 1\rceil)|\leq 0$; that is, $G$ has  $\lceil 2/k \rceil\leq \diam(G)$. Because $G$ is finite, $G$ has polynomial volume growth with $|\dist^{-1}_x(r)|\leq |V(G)|r^0$ (depending on $G$, a much tighter bound may be possible.)  
We develop results that are useful in the case that $G$ has a negative lower bound on curvature.  We find that such graphs do not necessarily have polynomial volume growth.  {\it We remark here that it remains an open question whether or not a bound of $\kappa(x,y) \geq 0$, for all $x,y\in V$, implies polynomial volume growth.}

\begin{theorem}\label{thm:vg1}Let $G$ be a $d$-regular graph with $\kappa(v_1,v_2)\geq k$, for every pair of vertices $v_1,v_2$.  Fix $x\in V$, define $S_i = \dist^{-1}_x(i)$.  Then for $i\geq 1$, $$|S_{i+1}|\leq \frac{d+1-2dk}{2}|S_i|.$$\end{theorem}

\begin{proof}
First, we bound $e(S_i,S_{i+1})$, the number of edges between $S_i$ and $S_{i+1}$. Let $z\in S_i$, $z$ is adjacent to some vertex $y(z)\in S_{i-1}$.  (If $z$ is adjacent to multiple vertices in $S_{i-1}$, choose $y(z)$ arbitrarily from them.)  Let $T(z)$ be the set of common neighbors of $z$ and $y(z)$. Neither $y$ nor a neighbor of $y$ can be in $S_{i+1}$, so $e(z,S_{i+1})\leq d-1-|T(z)|$.  (Note that we frequently suppress $y(z)$ to $y$.)  Let $T^* = \sum_{z\in S_i} |T(z)|$, so $e(S_i,S_{i+1})\leq (d-1)|S_i| - T^*.$

Next, for each $z$ we wish to use the Kantorovich characterization of $W_1(\mu_y,\mu_z)$.
Define the following test-function $f$:\begin{itemize} \item $f(y) = 0$. \item $f(z) = 1$. \item $f|_{T(z)} = 0$. \item For any other neighbor $v$ of $y$, $f(v) = -1.$  \item Let $W(z)$ be the set of neighbors of $z$ (besides $y$) that are not in $T(z)$ and are adjacent to a neighbor of $y$ (besides $z$) that is not in $T(z)$.  We may set $f|_{W(z)} = 0$. \item Let $U(z) = N(z) \setminus\parens{ \{y\}\cup T(z)\cup W(z)}$, set $f|_{U(z)} = 1$.  Here we use $N(z)$ to denote the set of neighbors of $z$.
\item $f$ can be made $1$-Lipschitz by setting $f = 0$ on every other vertex.
\end{itemize}
We have:
\[\sum_x f(x)\mu_z(x) = 1 + \frac{|U(z)|-d}{2d} \ \mbox{ and } \sum_x f(x)\mu_y(x) = \frac{|T(z)|+2-d}{2d}.\]
Combining the two, we get
\[(1-k)\geq \sum_x f(x)(\mu_z(x)-\mu_y(x)) \geq 1 + \frac{|U(z)|-2-|T(z)|}{2d}\,,\] and rearranging gives
\[|T(z)|+2-2dk \geq |U(z)|\,.\]

If a neighbor of $z$ is not in $U(z)$, the neighbor must be either $y$, adjacent to $y$ (and thus in $T(z)$), or adjacent to more than one neighbor of $y$, and hence in $W(z)$.\\ Any vertex in $S_{i+1}$ for which $z$ is the only neighbor in $S_i$ must be in $U(z)$.  The total number $U^*$ of vertices in $S_{i+1}$ that are adjacent to only one vertex in $S_i$ is at most \begin{align*}U^*\leq \sum_z |U(z)|\leq \sum_z \parens{|T(z)|+2-2dk} = T^*+|S_i|(2-2dk).\end{align*}

We can now see that the number of vertices in $S_{i+1}$ that are adjacent to more than one vertex in $S_i$ is bounded above by $$\frac{(d-1)|S_i|-T^*-U^*}{2}. $$  This is because the total number of possible edges from $S_i$ to these vertices is at most $e(S_i,S_{i+1})\leq (d-1)|S_i|-T^*$ less the $U^*$ edges that are accounted for by vertices in $S_{i+1}$ with only one neighbor in $S_i$.  Every other vertex must be incident to at least $2$ of those $(d-1)|S_i|-T^*-U^*$ edges, so we divide by $2$. 

Now, we add the other $U^*$ vertices in $S_{i+1}$ to achieve the desired result: 
\begin{align*}
|S_{i+1}|& \leq U^*+\frac{(d-1)|S_i|-T^*-U^*}{2} = \frac{(d-1)|S_i|-T^*+U^*}{2}\\
	&\leq\frac{(d-1)|S_i|-T^* + (2-2dk)|S_i|+T^*}{2} = \frac{d+1-2dk}{2}|S_i|.
\end{align*}
\end{proof}

Following the same proof outline, we obtain a better bound for bipartite graphs.

\begin{theorem}\label{thm:vg_bipartite}Let $G$ be a $d$-regular bipartite graph with $\kappa(v_1,v_2)\geq k$ for every pair of vertices $v_1,v_2$.  Fix $x\in V$, define $S_i = \dist^{-1}_x(i)$. For $i\geq 1$, $$|S_{i+1}|\leq \frac{d-dk}{2}|S_i|.$$\end{theorem}

\begin{proof}
First, we bound $e(S_i,S_{i+1})$, the number of edges between $S_i$ and $S_{i+1}$.  Let $z\in S_i$, $z$ is adjacent to some vertex $y(z)\in S_{i-1}$. Because $y\notin S_{i+1}$, $e(z,S_{i+1})\leq d-1$. Clearly, $e(S_i,S_{i+1})\leq (d-1)|S_i|$.

Next, for each $z$ we wish to use the Kantorovich characterization of $W_1(\mu_y,\mu_z)$.
Define a test-function $f$: 
\begin{itemize} \item $f(y) = 0$.
\item $f(z) = 1$.
\item  For any other neighbor $v$ of $y$, $f(v) = -1.$ 
\item Let $W(z)$ be the set of neighbors of $z$ (besides $y$) that are adjacent to a neighbor of $y$ other than $z$.  Set $f|_{W(z)} = 0$.  
\item Let $U(z) = N(z) \setminus \parens{ \{y\}\cup W(z)}$, set $f|_{U(z)} = 2$.  
\item $f$ can be made $1$-Lipschitz by setting $f = 1$ on any other vertex in the same set of the bipartition as $z$ and $f = 0$ on any other vertex in the same set as $y$. \end{itemize}
We have:
\[\sum_x f(x)\mu_z(x) = 1 + \frac{2|U(z)|-d}{2d} \ \mbox{ and } \ \sum_x f(x)\mu_y(x) = \frac{2-d}{2d}\,.\]
\\\text{Combining, }
\[(1-k)\geq \sum_x f(x)(\mu_z(x)-\mu_y(x)) \geq 1 + \frac{2|U(z)|-2}{2d}\,,\]\\ \text{resulting in }
\[ |U(z)| \le 1-dk\,.\]

If a neighbor of $z$ is not in $U(z)$, it must be either $y$ or adjacent to more than one neighbor of $y$, and thus in $W(z)$.\\ Any vertex in $S_{i+1}$ for which $z$ is the only neighbor in $S_i$ must be in $U(z)$.  The total number $U^*$ of vertices in $S_{i+1}$ that are adjacent to only one vertex in $S_i$ is at most \begin{align*}U^*\leq \sum_z |U(z)|\leq |S_i|(1-dk).\end{align*}

We can now bound the number of vertices in $S_{i+1}$ that are adjacent to more than one vertex in $S_i$ from above by $$\frac{(d-1)|S_i|-U^*}{2}. $$  This is because the total number of possible edges from $S_i$ to these vertices is at most $e(S_i,S_{i+1})\leq (d-1)|S_i|$ less the $U^*$ edges that are accounted for by vertices in $S_{i+1}$ with only one neighbor in $S_i$.  Each counted vertex must be incident to at least $2$ of those $(d-1)|S_i|-U^*$ edges, so we divide by $2$. 

Now, we add the other $U^*$ vertices to achieve the desired bound on $|S_{i+1}|$: \begin{align*}|S_{i+1}|\leq U^*+\frac{(d-1)|S_i|-U^*}{2} = \frac{(d-1)|S_i|+U^*}{2}\\ \leq\frac{(d-1)|S_i| + (1-dk)|S_i|}{2} = \frac{d(1-k)}{2}|S_i|.\end{align*}
\end{proof}

We continue to denote $S_i=\dist_x(i)$ and summarize the results of Theorems \ref{thm:vg1} and \ref{thm:vg_bipartite} as follows.
\begin{theorem} \label{thm:vg}
For any $d$-regular graph, for $i\geq 1$, we have \begin{align*}|S_i|\leq d^i\biggl(\frac{1+\tfrac{1}{d}-2k}{2}\biggr)^{i-1}.\end{align*} 

For any $d$-regular bipartite graph, for $i\geq 1$, we also have \begin{align*}|S_i|\leq d^i \biggl(\frac{1-k}{2}\biggr)^{i-1}.\end{align*}
\end{theorem}

\begin{proof}
Observe $S_0 =1$ and $S_1 = d$ for every graph. Repeated application of Theorems \ref{thm:vg1} and \ref{thm:vg_bipartite} gives the desired bounds for $S_i$ when $i\geq 2.$ This can be made formal using induction, which is left to the reader.
\end{proof}

As far as we are aware, these are the first non-trivial bounds on volume growth under a negative bound on Ollivier curvature.  A weakness in the proof method is that vertices in $S_{i+1}$ are counted either as having exactly one neighbor in $S_i$ (i.e., in some set $U(x)$), or as having several neighbors ($W(x)$), but the upper bound on $|S_{i+1)}|$ assumes the worst case - that there are a large number of vertices of type $W$, each having only $2$ neighbors in $S_i$.  For graphs where that assumption is correct (or close), our bound is somewhat tight.  In other graphs, the average number of neighbors in $S_i$ for any vertex in $S_{i+1}$ can be $O(d)$.  For those graphs the bound is not tight.  Below we give an example illustrating this issue.

\begin{example}
Let $T_p$ be the infinite $p$-regular tree and $T_p^q$ be the Cartesian product graph $T_p\square T_p\square \cdots \square T_p$, with the product taken $q$ times. Note that $T_p^q$ is $pq$-regular.  It is easy to compute that $T_p$ has $k(x,y) = \tfrac{2-p}{p}$ if $x\sim y$.  By tensorization of curvature (see for instance \cite{KKRT15}), we know that $T_p^q$ has $k(x,y) \geq \tfrac{2-p}{pq}$ whenever $x\sim y$.

Because $T_p^q$ is bipartite, we apply the second statement of Theorem \ref{thm:vg} to find the bound \begin{align}|d^{-1}_x(i)|\leq (pq)^i\parens{\frac{1-\tfrac{2-p}{pq}}{2}}^{i-1} = pq\parens{\frac{p(q+1)}{2}-1}^{i-1},\end{align} so that \begin{align} \log(|d^{-1}_x(i)|)\leq  i\log\parens{\frac{p(q+1)}{2}-1} + O(1).\end{align}

A vertex $y\in \dist^{-1}_x(i)$ is characterized by the distance from $x$ parallel to each of the $q$ copies of $T_p$ in the product graph, and, given those distances, by the path taken in $T_p$ of that distance.

There are $\binom{i+q-1}{q}$ choices of what distance is traveled along each copy of $T_p$. At each step of any path taken along some copy of $T_p$, there are either $p$ possibilities (for the first step) or $p-1$ possibilities (for any subsequent step).  As such, we have 
\begin{align} \binom{i+q-1}{q}(p-1)^i\leq |\dist^{-1}_x(i)|\leq \binom{i+q-1}{q}p^q(p-1)^{i-q}.\end{align} 
It follows that
\begin{align}\log(|\dist^{-1}_x(i)|)=i\log(p-1) + O(1).\end{align}

Observe that $q$ is the maximum number of neighbors that $y\in \dist^{-1}_x(i)$ has in $\dist^{-1}_x(i-1).$  The difference between the bound from Theorem \ref{thm:vg} of $i\log \parens{\frac{p(q+1)}{2}-1}$ and the actual value $i\log (p)$ results from the value of $q$.  As discussed before, the reason for this is that in the proof of Theorem \ref{thm:vg}, the upper bound assumes as a worst-case scenario that every vertex in $\dist^{-1}_x(i)$ has either $1$ or $2$ neighbors in $\dist^{-1}_x(i-1)$.  But in fact, as $i$ grows almost every vertex in the $i$-shell of $T_p^q$ has $q$ neighbors in the $(i-1)$-shell. 
\end{example}

{\it We conjecture here that $T_p^q$ actually experiences the maximum volume growth for their curvature and regularity.}
\begin{conjecture}Let $G$ be a $pq$-regular graph so that if $u,v\in V(G)$, then $\kappa(u,v)\geq \frac{2-p}{pq}$.  Let $x\in V(G)$ and $y\in V(T_p^q)$.  Then for any $i\geq 0$,  one has 
\[|\dist^{-1}_x(i)|\leq |\dist^{-1}_y(i)|.\]\end{conjecture}
Qualitatively, the graph $T_p^q$ is conjectured to fill the same role that the space of constant curvature does in the Bishop Volume Comparison Theorem.  A case of this conjecture is that the $d$-dimensional lattice $T_2^d$ is conjectured to have the fastest volume growth for any $2d$-regular graph with curvature lower bound $0$.  If correct, this would prove that any such graph has polynomial volume growth.
    
\section{Volume Growth and Spectral Estimates in Graphs}
\label{sect:Spectrum}
    
    In this section we follow the methods from the continuous setting that were developed by B. Benson \cite{B15} and discussed in Section \ref{sect:vg_lambda}.  First, we demonstrate an upper bound for an eigenvalue $\lambda_k(G)$ by taking the Rayleigh quotient of a function based only on distance from a cut-set $\Sigma$.  Next, we optimize that quotient by treating it as a discrete Hardy-type inequality.  

\begin{remark}  In applying our results using volume growth to bound the spectrum, we will use the relationship between notions of curvature of the graph and volume growth, as introduced and referenced in the previous section. The bounds which illustrate this relationship are the only point in our analysis that relies on the discrete curvature.  Given another volume growth result (either based on another notion of discrete curvature or unrelated to curvature), it will be possible to repeat the analysis we present here and achieve similar results.\end{remark}

\subsection{Bounding eigenvalues using volume growth}

In this section, we will establish bounds for the spectrum of the graph Laplacian using bounds on volume growth. 

Our methods in this section for approximating $\lambda_k$, where $k\geq 2$, do not make any assumption about the cut-set, but the bounds we obtain will only be in terms of the generic volume growth bounds $\mu,\nu$.  Later, we will give a bound for $\lambda_2$ with the assumption that $\Sigma$ is the outer vertex isoperimetric optimizing cut-set.

Let $G = (V,E)$ be a graph. Let $\Sigma\subset V(G)$ be a cut-set that separates $V\setminus\Sigma$ into $V^+$ and $V^-$.  Note that under this definition, it is possible that $V^+$ or $V^-$ is empty. The signed distance function $\dist_{\Sigma}:V\to\Z$ is defined so that $|\dist_{\Sigma}(v)| = \text{min}_{x\in \Sigma} \dist_G(x,v)$ where $\dist_G$ is the graph distance, and the sign of $\dist_{\Sigma}$ is positive on $V^+$ and negative on $V^-$. 

We will assume that we have volume growth and decay bounds for the level sets of $\dist_{\Sigma}.$ Specifically, let $\nu(k)$ denote a volume growth bound and $\mu (k)$ denote a uniform volume decay bound respectively. Here, for $k \in \Z,$ the bounds $\nu (k)$ and $\mu(k)$ have the property that 
\begin{equation}\label{eq:GrowthBound}
    |\Sigma|\mu (k)  \leq |\dist^{-1}_{\Sigma}(k)|\leq |\Sigma| \nu (k).
	\end{equation}
\begin{definition}\label{def:spaces} Define $T^+ \in \Z_{>0}$ so that $\mu (k)>0$ for all $0 \leq k \leq T^+$ and define $T^- \in \Z_{<0}$ so that $\mu (k) >0$ for all $T^- \leq k \leq 0$.

We denote by $U^+$ the set of all functions $f^+:V^+\to \R$ and $U^- = \{f^-:V^-\to \R\}$.  The inner products for $U^+$ and $U^-$ will be inherited from the inner product on $V$ by letting $f^+(v)=0$ for every $v \in V\setminus V^+$ and $f^-(v)=0$ for every $v \in V\setminus V^-,$ i.e. extending the functions on $U^+$ and $U^-$ to all of $V$ by zero, then taking the inner product on all of $V.$

We also define a pair of corresponding spaces of functions on $\{0,\dots, T^+\}$ and $\{T^-, T^-+1, \ldots,-1,0\}$: let $W^+$ be the space of functions $g^+:\{0,1,2,\ldots , T^+\} \to \R$ such that $g^+(0)=0$, and $W^-$ be the space of functions $g^-:\{T^-, T^-+1, \ldots,-1,0\} \to \R$ such that $g^-(0)=0.$ Further, for $u^+,v^+ \in W^+,$ define 
    $$\langle u^+,v^+\rangle_+=\sum_{i=0}^{T^+} u(i)v(i).$$
Similarly, for $u^-,v^- \in W^-,$ define 
    $$\langle u^-,v^-\rangle_-=\sum_{i=T^-}^{0} u(i)v(i).$$
\end{definition}
    
To estimate $\lambda_2(G),$ we will be interested in (the smallest positive) solutions $\rho^{+} \in \R^{T^{+}}$ and $\rho^- \in \R^{T^-}$ which satisfy the respective equations 
\begin{align}\label{eq:Hardy1}
\sum_{i=0}^{T^{+}} \phi(i)^2 \cdot \parens{\nu(i)+\nu(i-1)} &\leq
2\rho^{+} \sum_{k=0}^{T^{+}}\biggl [ \sum_{i=0}^k \phi(i) \biggr ]^2 \cdot \mu(k),\\
\label{eq:Hardy2}
\sum_{i=T^-}^0 \phi (i)^2 \cdot \parens{\nu(i)+\nu(i-1)} &\leq 2\rho^- \sum_{k=T^-}^0  
\biggl [ \sum_{i=k}^0 \phi(i) \biggr ]^2 \cdot \mu(k)
\end{align}
with $\phi(0)=0,$ $\phi \not\equiv 0$ and where $\nu$ and $\mu$ are the volume growth bounds defined in Equation \ref{eq:GrowthBound}. Equations of this form are called weighted discrete Hardy inequalities. For a fuller discussion of this topic, we refer to \cite{Miclo}.

\begin{theorem}\label{theo:eigen2}
	Let $\rho^+$ and $\rho^-$ be defined by Equations \ref{eq:Hardy1} and \ref{eq:Hardy2}. Then $\lambda_2(G) \leq \max \{\rho^+, \rho^-\}.$
\end{theorem}

Before proving the theorem, we formulate the results for the higher eigenvalues.
To estimate the higher eigenvalues, we define a symmetric, tridiagonal matrix $A^+$ indexed by $\{1,\dots, T^+\}$ so that 
\begin{align*}
A^+_{ij} = \begin{cases}
\tfrac{2\nu(i) + \nu(i-1)+\nu(i+1)}{\mu(i)}, &\text{ if }i=j;~ 1\leq i<T^+\\
\tfrac{\nu(T^+)+\nu(T^+-1)}{\mu(T^+)},&\text{ if }i = j = T^+\\
\tfrac{-\nu(i)-\nu(j)}{\sqrt{\mu(i)\mu(j)}},&\text{ if }|i-j| = 1;~ 1\leq i,j\leq T^+\\
0, &\text{otherwise}
\end{cases}
\end{align*}
Similarly, we define the symmetric, tridiagonal matrix $A^-$ indexed by $\{T^-, T^-+1, \ldots , -2, -1\}$:
\begin{align*}
A^-_{ij} = \begin{cases}
\tfrac{2\nu(i) + \nu(i-1)+\nu(i+1)}{\mu(i)}, &\text{ if }i=j; T^-<i\leq -1\\
\tfrac{\nu(T^-)+\nu(T^-+1)}{\mu(T^-)}, &\text{ if }i = j = T^-\\
\tfrac{-\nu(i)-\nu(j)}{\sqrt{\mu(i)\mu(j)}},&\text{ if }|i-j| = 1; T^- \leq i,j\leq -1\\
0, &\text{otherwise.}
\end{cases}
\end{align*}

\begin{theorem}\label{theo:eigenhigher}
	For a graph $G$ and any $k,l \in \N$ with $1 \leq k,l \leq \min \{\left |T^-\right |, T^+\}=:T,$ then  
\begin{equation}\label{eq:eigenhigher}
        \lambda_{k+l} (G) \leq \frac{1}{2}\max \left \{ \rho_k^+
        	, \rho_l^- \right \}
        \end{equation}
    where $\rho^+_k$ and $\rho^-_l$ are the $k$-th and $l$-th non-trivial eigenvalues of the respective equations $A^+g^+=\rho^+g^+$ with $g^+ \in W^+$ and $A^-g^-=\rho^-g^-$ with $g^- \in W^-.$ 
    
    In particular, we have that 
    	\begin{equation*}
        	\lambda_j(G) \leq \frac{1}{2}\min_{j=k+l}
            \min_{\substack{j\leq t^++|t^-|\leq 2T\\t^+,-t^- \geq 1}}
            \max \left \{ 
            \rho^+_k(t^+), 
            \rho_l^-(|t^-|)
            \right \}
        \end{equation*}
where $\rho^+_k(t^+)$ and $\rho_l^-(|t^-|)$ are the eigenvalues of the symmetric, tridiagonal matrices $A^+(t^+)$ and $A^-(t^-)$ resulting from replacing $T^+$ with $t^+$ and $T^-$ with $t^-$ in the definitions of $A^+$ and $A^-,$ respectively.
\end{theorem}

\begin{remark}
Broadly speaking, we are using estimates of volume growth and decay which act as weights and linearize the graph  Laplacian eigenvalue problem on the graph. The main idea is to linearize the graph using the weights $\nu$ and $\mu.$ This can be done in multiple ways, however, we choose to do it in a way which is based around a minimizing vertex cut for $h_{out}(G),$ as we find it to be a natural way to produce Buser-type inequalities for both $\lambda_2(G)$ as well as the higher eigenvalues and higher Cheeger constants. In this form, we show that the higher eigenvalues are bounded above by eigenvalues of tridiagonal matrices, which in some cases, are known in closed form \cite{KST99}.
\end{remark}

  We will now prove both Theorems \ref{theo:eigen2} and  \ref{theo:eigenhigher} simultaneously.

\begin{proof}[Proof of Theorems \ref{theo:eigen2} and \ref{theo:eigenhigher}]
Using the Poincar\'e minimax principle for characterization of eigenvalues, we see that
\begin{equation}\label{eq:Ray1}
	\lambda_k (G)= \inf_{U_k}\sup_{f\in U_k}
	\frac{\langle f,\Delta f\rangle }{\langle f,f\rangle },
\end{equation} 
where $U_k$ is the set of $k$-dimensional subspaces of functions $f \in \R^V.$ Expanding these inner products, we find that
	\begin{align*}
\langle f,\Delta f\rangle &= 
\sum_x f(x)\sum_{y\sim x}\tfrac{1}{d}\parens{f(x)-f(y)}\\ 
	&= \sum_{\{x,y\}:x\sim y}\tfrac{1}{d}\parens{f(x)-f(y)}^2\\ 
    &= \sum_x\tfrac{1}{2d}\sum_{y\sim x}\parens{f(x)-f(y)}^2, \text{ and}\\
\langle f,f\rangle &= \sum_x f^2(x).
\end{align*}

Define $\Sigma \subset V$ so that $h_{out}(G)=|\Sigma|/|A|$ where $A\subset V,$ with $\Sigma=\partial_{out} A,$ and $|A|\leq |V|/2.$ Use the signed distance from $\Sigma,$ with positive distance into $A,$ to define the following vertex subsets:
    \begin{equation*}
        V_{\geq a}:=\bigcup_{n \in \N: n\geq a} \dist^{-1}_{\Sigma}(n) \text{ and }
        V_{\leq a}:=\bigcup_{n \in \N: n\leq a} \dist^{-1}_{\Sigma}(n).    
    \end{equation*}
   We wish to estimate the eigenvalue $\lambda_j (G)$ of the Laplacian on $G$ by the eigenvalues of the Laplacian on the subgraphs on $V_{\geq 0}$ and $V_{\leq 0}$, wherein the test functions for the Rayleigh quotients  $f^+:V_{\geq }\to \R$ satisfy $f^+(v)=0$ for every $v\in \dist_{\Sigma}^{-1}(0)=\Sigma$, and $f^-:V_{\leq 0} \to \R$ satisfy $f^-(v)=0$ for every $v \in \Sigma$.  Denote the eigenvalues determined by the Rayleigh quotients of these specific test functions by $\xi_k(V_{\geq 0})$ and $\xi_l(V_{\leq 0})$, respectively. Using the Poincar\'e minimax principle it is possible to see that when $1 \leq k,l\leq \min \left \{ |V^-|,|V^+|\right \},$ it follows that 
	\begin{equation}\label{eq:Poincare}
    \lambda_{k+l} (G) \leq 
    	\max \left \{ \xi_k(V_{\geq 0}), \xi_l(V_{\leq 0}) \right \}.
	\end{equation}
This can be seen by discretizing Theorem 8.2.1 found in \cite{Buser} or Proposition 2.1 of \cite{B15}. Such an argument is given in detail by Balti for weighted directed graphs, where the result is also extended in several ways, including to the special Laplacian operator on these graphs \cite[Section 5]{Balti}.

To give an upper bound for the eigenvalue, we restrict the test functions for Equation \ref{eq:Ray1} in $\R^V$ to functions which 
	\begin{enumerate}
		\item vanish on either $V_{\geq 0}$ or $V_{\leq 0},$ 
        \item are constant on each set $\dist^{-1}_{\Sigma}(i),$ 
        \item and are constant for values less than or equal to $T^-$ and greater than or equal to $T^+.$ 
	\end{enumerate}
Recall the definitions of the spaces of functions $U^{+}, U^{-},$ and $W^+, W^-$ from Definition \ref{def:spaces}. We will first treat these test functions as functions in $U^+$ or $U^-,$ respectively, and then as functions in $W^+$ or $W^-$. Let $U^+_k$ and $U^-_l$ be, respectively, arbitrary sets of $k$- and $l$-dimensional subspaces of real-valued functions of $U^+$ and $U^-$, for values $k,l \in \Z_{\geq 0}.$  Similarly, $W^+_k$ and $W^-_l$ are $k$- and $l$-dimensional subspaces of $W^+$ and $W^-$, respectively.  

Combining Equation \ref{eq:Poincare} with the Poincar\'e minimax characterization for the Dirichlet eigenvalues $\xi_k(V_{\geq 0})$ and $\xi_l(V_{\leq 0})$ while maintaining the assumption that $1\leq k,l\leq \min\{|V^-|,|V^+|\}$ we have that 
\begin{equation}
\label{eq:Ray_higher}
	\lambda_{k+l}(G) \leq \max \left \{\inf_{U_k^+}\sup_{f^+\in U_k^+}
	\frac{\langle f^+,\Delta f^+\rangle }{\langle f^+,f^+\rangle }, 
    \inf_{U_l^-}\sup_{f^-\in U_l^-}
	\frac{\langle f^-,\Delta f^-\rangle }{\langle f^-,f^-\rangle } \right \}.
\end{equation}

Now, for $g$ defined on $V_{\geq 0},$ with constant value $g(i)$ on the $i$-shell, using the volume growth estimates from Equation \ref{eq:GrowthBound}, we have the estimates 
	\begin{align*}
    	\langle g,\Delta g\rangle &= \sum_{i=0}^{T^+}\sum_{x\in \dist^{-1}_{\Sigma}(i)} \sum_{y\sim x}\frac{1}{2d}\parens{g(i)-g(y)}^2\\ &\leq \frac{1}{2}\sum_{i=0}^{T^+}\sum_{x \in \dist_{\Sigma}^{-1}(i)}\biggl[ \parens{g(i)-g(i+1)}^2 + \parens{g(i)-g(i-1)}^2\biggr]\\ &=\frac{1}{2}\sum_{i=0}^{T^+}|\dist^{-1}_{\Sigma}(i)|\biggl[ \parens{g(i)-g(i+1)}^2 + \parens{g(i)-g(i-1)}^2\biggr]\\
        &\leq \frac{1}{2}\sum_{i=0}^{T^+}\nu(i)|\Sigma|\biggl[ \parens{g(i)-g(i+1)}^2 + \parens{g(i)-g(i-1)}^2\biggr]\end{align*}
and \begin{align*}\langle g,g\rangle &= \sum_{i=0}^{T^+}|\dist^{-1}_{\Sigma}(i)|g^2(i)\\
	& \geq \sum_{i=0}^{T^+}\mu(i)|\Sigma| g^2(i).
	\end{align*}
Similar estimates hold for a function $g$ defined on $V_{\leq 0}.$

We now use these bounds in Equation \ref{eq:Ray_higher}.  Because we are restricting to functions $g$ with a constant value in each shell, it is equivalent to write the expression in terms of $W^+$ and $W^-$, rather than $U^+$ and $U^-$.
\begin{equation}\label{eq:Ray3}
\begin{split}
\lambda_{k+l} (G) &\leq \max \left \{ \inf_{W_k^+}\sup_{g\in W_k^+}\frac{\sum_{i=0}^{T^+}\biggl[ \parens{g(i)-g(i+1)}^2 + \parens{g(i)-g(i-1)}^2\biggr]\cdot \nu(i)}{2\sum_{i=0}^{T^+}g^2(i)\cdot \mu(i)}, \right.\\
&\qquad \qquad \left. \inf_{W_l^-}\sup_{g\in W_l^-}\frac{\sum_{i=T^-}^{0}\biggl[ \parens{g(i)-g(i+1)}^2 + \parens{g(i)-g(i-1)}^2\biggr]\cdot \nu(i)}{2\sum_{i=T^-}^{0}g^2(i)\cdot \mu(i)}
\right \}.
\end{split}
\end{equation}
Note that $|\Sigma|,$ a common factor appearing in the numerator and denominator, has been eliminated in the resulting estimates in Equation \ref{eq:Ray3}. 

For estimating $\lambda_2(G),$ we take $k=l=1$ and the Rayleigh quotient for $W^+$ in Equation \ref{eq:Ray3} becomes
	\begin{equation}\label{eq:HardyRay}
    	\inf_{g \in W^+, g\not\equiv 0} \, \frac{\sum_{i=0}^{T^+}\biggl[ \parens{g(i)-g(i+1)}^2 + \parens{g(i)-g(i-1)}^2\biggr]\cdot \nu(i)}{2\sum_{i=0}^{T^+}g^2(i)\cdot \mu(i)}. 
	\end{equation}
Define $\phi(j):=g(j)-g(j-1)$ for the $W^+$ quotient in Equation \ref{eq:Ray3}. Noting that $g(j)=\sum_{i=0}^j \phi(i)$ and using the facts that $g(0)=0,$ $\phi(0)=0,$ and $g(i)$ is constant for all $i \geq T^+$ implies $\phi(i)=0$ for all $i > T^+.$ As a result, Equation \ref{eq:HardyRay} becomes 
    \begin{equation*}
        \inf_{g \in W^+, g\not\equiv 0} \,
        \frac{\sum_{i=0}^{T^+}[\phi(i+1)^2+\phi(i)^2]\nu(i)}{2\sum_{k=0}^{T^+}\left [\sum_{i=0}^k \phi(i) \right ]^2\mu(i)}
        =\inf_{g \in W^+, g\not\equiv 0} \,
        \frac{\sum_{i=0}^{T^+}\phi(i)^2[\nu(i)+\nu(i-1)]}{2\sum_{k=0}^{T^+}\left [\sum_{i=0}^k \phi(i) \right ]^2\mu(i)}
    \end{equation*}
In the rightmost equality, we have used that $\phi(T^++1)=g(T^++1)-g(T^+)=0.$

It follows from the definition of $\phi$ and a routine computation that the quotient in Equation \ref{eq:HardyRay} is bounded from above by $\rho^+$ in Equation \ref{eq:Hardy2}. A similar argument verifies $\rho^-$ in Equation \ref{eq:Hardy2}. This establishes Theorem \ref{theo:eigen2}. \hfill $\Box$

\medskip

\textbf{Bounding the higher eigenvalues:}
We now continue the argument for higher eigenvalues. Since the test function $g$ can be thought of as a test function vanishing off of $V_{\geq 1},$ we wish to find a symmetric matrix $A^+$ and a column vector $\vw_+$ so that 
$$\langle \vw_+,A^+\vw_+ \rangle_+=\sum_{i=0}^{T^+}\left [(g(i)-g(i+1))^2+(g(i)-g(i-1))^2 \right ]\nu(i)$$ 
and 
$$\langle \vw_+,\vw_+ \rangle_+=\sum_{i=0}^{T^+} g(i)^2\mu(i).$$
    As a result, the quotient $\tfrac{\langle \vw_+,A^+\vw_+ \rangle_+}{2\langle \vw_+,\vw_+ \rangle_+}$ is equal to the eigenvalue estimate for $V_{\geq 0}$ in Equation \ref{eq:Ray3}.

Since $g(0)=0,$ we omit the $i=0$ entry in constructing the column vector $\vw_+,$  defining 
    $$\vw_+= \begin{bmatrix}g(1)\sqrt{\mu(1)}\\ \vdots\\ g(T^+)\sqrt{\mu(i)}\end{bmatrix}.$$
In other words, we let the $i$-th entry of $\vw_+$ to be $g(i)\sqrt{\mu(i)}.$    
We now wish to construct a symmetric matrix $A^+$ such that 
    \begin{equation}\label{eq:Aplus}
        \vw_+^{\intercal}A^+\vw_+=\langle \vw_+,A^+\vw_+\rangle_+
        =\sum_{i=0}^{T^+}[(g(i)-g(i+1))^2+(g(i)-g(i-1))^2]\nu(i),
    \end{equation}
where $\vw_+^{\intercal}$ is the transpose of $\vw_+$ when written as a column vector.

Now, the $i$-th term in the right hand side of Equation \ref{eq:Aplus} can be rewritten as 
    \begin{equation} \label{eq:gForm}
        [2g(i)^2+g(i-1)^2+g(i+1)^2-2g(i)(g(i-1)+g(i+1))]\nu(i).
    \end{equation}
It follows from Equation \ref{eq:gForm} that the entries of $A^+$ are quotients where the numerator can be expressed as linear combinations of the weights $\nu$ and the denominator of $A^+_{ij}$ is equal to $\sqrt{\mu(i)\mu(j)}.$ Since $A^+_{ij}$ and $A^+_{ji}$ correspond to the right hand side of Equation \ref{eq:Aplus}, the presence of terms of the form 
    $$-2g(i)g(i-1)\nu(i)=g(i)g(i-1)(-2\nu(i))$$
contribute an additive term $-\nu(i)$ to the numerator of each entry $A_{i-1,i}^+$ and $A_{i,i-1}^+$, where the factor of $\nu(i)$ has been halved due to the fact that we require $A^+_{i-1,i}=A^+_{i,i-1}.$ For the same reason, terms of the form 
    $$-2g(i)g(i+1)\nu(i)=g(i)g(i+1)(-2\nu(i))$$
contribute an additive term $-\nu(i)$ to the numerator of $A_{i,i+1}^+$ and $A_{i+1,i}^+.$ This implies that when $|i-j|=1,$ we have that $$A_{ij}^+=\frac{-\nu(i)-\nu(j)}{\sqrt{\mu(i)\mu(j)}}.$$
When $1\leq i<T^+,$ the terms 
    $$[2g(i)^2+g(i-1)^2+g(i+1)^2]\nu(i)$$
contribute $2\nu(i)$ to $A_{ii}^+$ and $\nu(i)$ to $A_{i-1,i-1}^+$ and $A_{i+1,i+1}^+,$ giving 
    $$A_{ii}^+=\frac{2\nu(i)+\nu(i-1)+\nu(i+1)}{\mu(i)}.$$
In other words, the numerator of the entry $A_{ii}^+$ contains the multiple of $g(i)^2$ in the right hand side of Equation \ref{eq:Aplus}, while the numerator of the entry $A_{ij}^+$ contains one half of the multiple of $g(i)g(j)=g(j)g(i)$ in the sum, since $A_{ij}^+=A_{ji}^+.$
Finally, the $T^+$-th term in the sum on right hand side of Equation \ref{eq:Aplus} can be written as 
    $$[g(T^+)^2-2g(T^+)g(T^+-1)+g(T^+-1)^2]\nu(T^+).$$
This contributes $-\nu(T^+)$ to the numerators of $A_{T^+,T^+}^+$ and $A_{T^+-1,T^+-1}^+.$ However, since $g(T^+)-g(T^++1)$ vanishes, we have that 
    $$A_{T^+,T^+}^+=\frac{\nu(T^+)+\nu(T^+-1)}{\mu(T^+)}.$$
Thus, we conclude that the desired symmetric matrix $A^+$ can be constructed as  
\begin{align*}
A^+_{ij} = \begin{cases}
\tfrac{2\nu(i) + \nu(i-1)+\nu(i+1)}{\mu(i)}, &\text{ if }i=j;~ 1\leq i<T^+\\
\tfrac{\nu(T^+)+\nu(T^+-1)}{\mu(T^+)},&\text{ if }i = j = T^+\\
\tfrac{-\nu(i)-\nu(j)}{\sqrt{\mu(i)\mu(j)}},&\text{ if }|i-j| = 1;~ 1\leq i,j\leq T^+\\
0, &\text{otherwise}
\end{cases}
\end{align*}
where $1\leq i,j\leq T^+.$

Since the test functions corresponding to the eigenvalues $\xi_l(V_{\leq 0})$ also vanish for vertices in $\Sigma,$ the arguments for finding the matrix $A_{ij}^+$ can be repeated to find that 
\begin{align*}
A^-_{ij} = \begin{cases}
\tfrac{2\nu(i) + \nu(i-1)+\nu(i+1)}{\mu(i)}, &\text{ if }i=j; T^-<i\leq -1\\
\tfrac{\nu(T^-)+\nu(T^-+1)}{\mu(T^-)}, &\text{ if }i = j = T^-\\
\tfrac{-\nu(i)-\nu(j)}{\sqrt{\mu(i)\mu(j)}},&\text{ if }|i-j| = 1; T^- \leq i,j\leq -1\\
0, &\text{otherwise.}
\end{cases}
\end{align*}
We can now estimate Equation \ref{eq:Ray3} from above using the matrices $A^+$ and $A^-$:
\begin{equation}\label{eq:Ray4}
 \lambda_{k+l}(G)\leq \frac{1}{2}\max \left \{ \inf_{W_k^+}\sup_{\vw_+\in W_k^+}\frac{\langle \vw_+,A^+\vw_+ \rangle_+}{\langle \vw_+,\vw_+ \rangle_+}, 
  \inf_{W_l^-}\sup_{\vw_-\in W_l^-}\frac{\langle \vw_-,A^-\vw_- \rangle_-}{\langle \vw_-,\vw_-\rangle_-}
\right \}.
\end{equation}

Since $A^+$ and $A^-$ are symmetric, the spectral theorem implies that there exist an orthonormal basis of $T^+$ real eigenfunctions of $A^+$ in $W^+$ with corresponding real eigenvalues and an orthonormal basis of $T^-$ real eigenfunctions of $A^-$ in $W^-$ having real eigenvalues. It is easy to see that if $\vw_{\ast}\in W^+$ is an eigenfunction of $A^+$ with corresponding eigenvalue $\rho_{\ast},$ we have 
	\begin{equation}\label{eq:SpecThm}
    	\frac{\langle \vw_{\ast},A^+\vw_{\ast} \rangle_+}
        	{\langle \vw_{\ast},\vw_{\ast} \rangle_+}=\rho_{\ast}.
	\end{equation}
The same relationship holds for $A^-$ and the eigenfunctions of $W^-.$
Since these bases of eigenfunctions are orthonormal, the $k$-th eigenvalue of $A^+$ in $W^+_k$ and the $l$-th eigenvalue of $W_l^-$ which we denote $\rho_k^+$ and $\rho_l^-,$ respectively, satisfy the following:
	\begin{equation}\label{eq:Spec2}
    	\inf_{W_k^+}\sup_{\vw_+\in W_k^+}\frac{\langle \vw_+,A^+\vw_+ \rangle_+}{\langle \vw_+,\vw_+ \rangle_+}=\rho_k^+ \hspace{.5in}\text{ and }
    	\hspace{.5in}
    	\inf_{W_l^-}\sup_{\vw_- \in W_l^-}\frac{\langle \vw_-,A^-\vw_- \rangle_-}{\langle \vw_-,\vw_- \rangle_-}=\rho_l^-
    \end{equation}
Combining Equations \ref{eq:Ray4} and \ref{eq:Spec2}, it follows that 
    \begin{equation*}
        \lambda_{k+l}(G) \leq 
        \frac{1}{2} \max \left \{\rho_k^+,\rho_l^- \right \}.
    \end{equation*}
This establishes Equation \ref{eq:eigenhigher}.    

The eigenvalue estimate on $\lambda_j(G)$ holds by taking $j=k+l$ in Equation \ref{eq:eigenhigher} while noting that the arguments above hold for any $t^+$ and $t^-$ with $1\leq t^+ \leq T^+$ and $T^-\leq t^- \leq -1,$ where one must restrict $k$ and $l$ such that $k \leq t^+$ and $l \leq |t^-|.$
\end{proof}

We remark that in the continuous case, one shows that analogue of the operator $A$ can be rewritten as a Sturm-Liouville problem depending on the same parameters of the manifold as Buser's inequality.  The details can be found in Benson \cite{B15}.

\subsection{Applying volume growth bounds}

In this section, we use $\nu(k)$ to denote a volume growth bound around $\Sigma$; i.e., a function with the property that, given a fixed $\Sigma\subset V$, all choices of sets $V^+,V^-$, and all $k\geq 0$, we have that $|\dist^{-1}_{\Sigma}(k)|\leq |\Sigma| \nu(k)$. The function $\nu$ may depend on $\Sigma$ as well as the curvature, though previously we have only presented volume growth bounds that are independent of the choice of $\Sigma$. For convenience, we often denote $\Sigma_k=\dist_{\Sigma}^{-1}(k).$

\begin{remark}
In this section our results are in terms of the outer vertex isoperimetric constant $h_{out}$.  This is most natural because we use the counting measure on the vertex set. As stated before, there are simple bounds relating $h_{out}$ to the edge isoperimetric constant $h$:  $$h\leq h_{out}\leq hd\,,$$ where $d$ is the degree of the graph. Using these inequalities, it is possible to rewrite our results in terms of $h$. 
\end{remark}

\begin{lemma}\label{lem:d_lowerbound} Let $A\subset V$ be the set that achieves the outer vertex isoperimetric constant $h_{out}$ and let $\Sigma = \partial_{out}A$.  Set either $V^+ = A$ or $V^+ = V\setminus (A\cup \Sigma)$, and let $V^-$ be the other. Use this choice of $V^{+}$ and $V^-$ to define the signs (positive and negative, respectively) of the signed distance function $\dist_{\Sigma}$.  Let $k\geq 0,$ for $\Sigma_k = \dist^{-1}_{\Sigma}(k),$ it follows that  
\begin{align*} |\Sigma_k| \geq |\Sigma |\biggl(1 -h_{out}\sum_{i=0}^k \nu(i)\biggr). 
\end{align*}
\end{lemma}

\begin{proof}
Observe that the case $k = 0$ is trivial.  Assume $k>0$.

Define $C^- = \bigcup_{i < k}\dist^{-1}_{\Sigma}(i)$ and $C^+ = \bigcup_{i > k}\dist^{-1}_{\Sigma}(i).$ We will split the proof into two cases.

\begin{enumerate}
\item In the first case, suppose $|C^-|< \tfrac{1}{2}|V|$.  Since $k>0$, we have that $(V^-\cup \Sigma )\subseteq C^-$, \\
therefore $|V^-\cup \Sigma|< \tfrac{1}{2}|V|.$  By assumption $\tfrac{1}{2}|V| \leq |V\setminus A| = \left|\big (V\setminus (A\cup\Sigma)\big )\cup\Sigma \right|,$
so $V^-\neq V\setminus (A\cup\Sigma).$ It follows that $V^- = A$.

Because $|A| =|V^-|<|C^-|< \tfrac{1}{2}|V|$, we have that 
$$\frac{|\Sigma|}{|A|}=h_{out} \leq \frac{|\Sigma_k|}{|C^-|}$$ and so 
$|\Sigma_k|\geq h_{out}|C^-|\geq h_{out}|A| = |\Sigma|$ and the result follows.\\

\item In the other case, we have $|C^-|\geq \tfrac{1}{2}|V|.$ Because $C^-$ and $C^+$ are disjoint,\\ we have that $|C^+|\leq \tfrac{1}{2}|V|$.  Therefore 
\begin{align*}|\Sigma_k|\geq h_{out}|C^+| = h_{out}\left (|V^+|-\sum_{i=1}^k |\dist^{-1}_{\Sigma}(i)|\right ).
\end{align*}

Observe that since $|A|\leq |V|/2,$ we have that 
\begin{align*}
|V^+|&\geq \min\{|V\setminus(A\cup\Sigma)|,|A|\}\\ 
	&\geq \min\{|V\setminus A|-|\Sigma|,|A|\}\\
    &\geq \min\{|A|-|\Sigma|,|A|\} \\
    &= |A|-|\Sigma|.
\end{align*}
Applying the previous bound gives us 
\begin{align*}|\Sigma_k|\geq h_{out}\biggl(|V^+|-\sum_{i=1}^k |\dist^{-1}_{\Sigma}(i)|\biggr) \geq h_{out}\biggl(|A|-|\Sigma|-\sum_{i=1}^k |\Sigma|\nu(i) \biggr)\\ = h_{out}\biggl(|A|-\sum_{i=0}^k |\Sigma|\nu(i) \biggr) = |\Sigma |\biggl(1 -h_{out}\sum_{i=0}^k \nu(i)\biggr),
\end{align*}

where the first equality relies on the (always reasonable) assumption that $\nu(0) \geq 1.$ This proves the result.
\end{enumerate}
\end{proof}

By this conclusion of Lemma \ref{lem:d_lowerbound}, one can think of the lower weights for vertex expansion $\mu(k)$ as 
    $$\mu(k)=1-h_{out}\sum_{i=0}^k \nu(i).$$
As a result, we have 
$$|\Sigma|\nu(k)\geq |\dist^{-1}_{\Sigma}(k)|\geq |\Sigma|\left (1-h_{out}\sum_{i=0}^k\nu(i)\right )$$ and from the Rayleigh quotient in Equation \ref{eq:Ray3}, we obtain
\begin{equation}\label{eq:Eigen}
\lambda_2 \leq \inf_{W_1}\sup_{g\in W_1}\frac{\frac{1}{2}\sum_{k=0}^{T}\nu(k)\biggl[ \parens{g(k)-g(k+1)}^2 + \parens{g(k)-g(k-1)}^2\biggr]}{\sum_{k=0}^{T}g^2(k)\left (1-h_{out}\sum_{i=0}^k \nu(i)\right )}.\end{equation}
where $T$ is the largest integer for which $1>h_{out}\sum_{i=0}^T \nu(i)$.
Here, by assumption we have the same volume growth bounds on $V^+$ and $V^-$, so (unlike the previous section) the Rayleigh quotients are identical on both sides of the cut-set.

\subsection{Bounds on $\lambda_2$}

Of particular interest is the problem of bounding $\lambda_2$.  Indeed, the original proofs of Buser's inequality only bound $\lambda_2$ and not the higher eigenvalues $\lambda_k:k\geq 3$. \cite{Buser82,Ledoux94,Ledoux04spectralgap}.

First, we will give a short proof of a bound on $\lambda_2$ that is independent of the Cheeger cut-set.

\begin{theorem}\label{thm:biglevel}
Let $\Sigma\subset V$ be a set (not necessarily the Cheeger-achieving cut set) that cuts $V$ into $V^+$ and $V^-$, and define the one-sided shells $\dist^{-1}_{\Sigma}(k)$ as before.  Let $\alpha = |\Sigma|/|V|$ .  Assume that $\alpha < 1/4$.  If $|\Sigma|\geq |\dist^{-1}_{\Sigma}(k)|$ for all $k\in \Z$, then $\lambda_2 \leq 8\alpha^2+ o(\alpha^2)$.
\end{theorem}

The proof loosely follows the method of the original proof of Buser's inequality for graphs.

\begin{proof}
Recall the Rayleigh quotient \begin{align*}\lambda_2(G) = \inf_f \frac{\tfrac{1}{2d}\sum_x \sum_{y\sim x} \parens{f(x)-f(y)}^2}{\sum_x f(x)^2}.\end{align*}
Without loss of generality assume that $|V^+| \geq |V^-|$. Let $t = \lfloor \frac{1}{4\alpha} \rfloor$  Because $\alpha < 1/4$ and $t>0,$ we can construct the following test-function in the Rayleigh quotient to bound $\lambda_2(G)$:
\begin{align*}
f(x) = 
\begin{cases}0 \text{ if } x\in \dist^{-1}_{\Sigma}(i)\text{ where } i\leq 0,\\
i \text{ if } x\in \dist^{-1}_{\Sigma}(i)\text{ where } 0\leq i \leq t,\\
t \text{ if } x\in \dist^{-1}_{\Sigma}(i)\text{ where } i \geq t.
\end{cases}
\end{align*}
For a vertex $x$,
\begin{align*}
\tfrac{1}{2d} \sum_{y\sim x} \parens{f(x)-f(y)}^2 \leq 
\begin{cases}
\tfrac{1}{2} \text{ if } x\in \dist^{-1}_{\Sigma}(i) \text{ where }0\leq i \leq t,\\
0 \text{ otherwise},
\end{cases}
\end{align*} and 
\begin{align*}
f(x)^2 \geq
\begin{cases}
t^2 \text{ if } x\in \dist^{-1}_{\Sigma}(i) \text{ where }i \geq t,\\
0 \text{ otherwise}.
\end{cases}
\end{align*}

Using these bounds, we see that 
\begin{align*}\sum_x \tfrac{1}{2d} \sum_{y\sim x} \parens{f(x)-f(y)}^2\leq \tfrac{1}{2}\sum_{i=0}^t |\dist^{-1}_{\Sigma}(i)|\leq \frac{(t+1)}{2}|\Sigma|
\end{align*} 
and 
\begin{align*}
\sum_x f(x)^2 \geq t^2 \sum_{i \geq t }|\dist^{-1}_{\Sigma}(i)| \geq t^2 \parens{|V^+|-t|\Sigma|} \geq t^2 \parens{\tfrac{1}{2}|V|-\tfrac{1}{4}|V|} = \tfrac{1}{4}t^2|V|.
\end{align*}
Combining the previous two inequalities, we find the result: \begin{align*}\lambda_2\leq \frac{2(t+1)|\Sigma|}{t^2|V|} = \parens{2/t + o(1/t)}\alpha = 8\alpha^2 + o(\alpha^2).\end{align*}
\end{proof}

Now we attempt to bound $\lambda_2$ in terms of the Cheeger cut-set in order to achieve a Buser-type result.
Observe that the Rayleigh minimizing function for $\lambda_2$ must have certain properties.

\begin{lem}\label{lem:gMonotone}
	The function $g(k)$ corresponding to the non-constant minimizer of the Rayleigh quotient in Equation \ref{eq:Eigen} is monotone in $k$.
\end{lem}

\begin{proof}[Proof of Lemma \ref{lem:gMonotone}.] We will induct on $k$. The base case is trivial since $g(0)=0$ by the Dirichlet boundary condition on $f^{-1}(0)$. Without loss of generality, assume that $g(1)\geq 0$, else replace $g(1)$ with $-g(1)$ and proceed to the induction step.

Assume for contradiction that $g$ is monotone increasing up to some $k$ in its domain, but that $g(k+1)<g(k)$. Then replacing $g(k+1)$ by $2g(k)-g(k+1)$, the numerator of $R(g)$ is unchanged as  $$\left [g(k)-\big (2g(k)-g(k+1)\big )\right ]^2=\big (g(k)-g(k+1)\big )^2.$$ At the same time, the denominator of $R(g)$ increases since $\big (2g(k-1)-g(k)\big )^2>g(k-1)^2$, therefore the quotient $R(g)$ decreases, contradicting the assumption that $g$ is a non-constant minimizer of $R(g)$.
\end{proof}

We are now able to bound the Rayleigh quotient within a constant factor.
To bound $\lambda_2$, we apply Equation \ref{eq:Eigen} giving the Rayleigh quotient \begin{align}\label{eqn:Ray4}
\lambda_2 \leq R:= 
\inf_f\frac{\frac{1}{2}\sum_{k=0}^{T}\nu(k)\biggl[ \parens{f(k)-f(k+1)}^2 + \parens{f(k)-f(k-1)}^2\biggr]}{\sum_{k=1}^{T}f^2(k)\left (1-h_{out}\sum_{i=0}^k\nu(i)\right )},
\end{align}
where the infimum is taken over all functions $f:\Z\to\R$ with $f(0) = 0$, $f(1)\neq 0$, $f(i) = 0$ if $i<0$ and $f(i) = f(T)$ if $i>T$.

\begin{theorem}\label{thm:vgB}The bounds on $R(g)$ are 
$\frac{1}{8B}\leq R\leq \frac{1}{2B},$ where $$B = \sup_{n\geq 1}\left (\sum_{k=n}^T (1-h_{out}\sum_{i=0}^k \nu(i))\right )\left (\sum_{k=1}^n\frac{1}{\nu(k)+\nu(k-1)}\right ).$$
\end{theorem}

\begin{proof} To apply a result of L. Miclo \cite{Miclo}, we write Equation \ref{eqn:Ray4} in a different form: set $g(k) = f(k)-f(k-1)$ for $k\in \Z$.  Observe that $f(k) = \sum_{i=1}^k g(i)$.  Also observe that $g(k) = 0$ if $k\leq 0$ or $k>T$.  We have
\begin{align*}
2R = 
\inf_g\frac{\sum_{k=0}^{T}\nu(k)\left [g(k+1)^2 + g(k)^2\right ]}
{\sum_{k=1}^{T}\biggl(\sum_{i=1}^k g(i) \biggr)^{\!\!2}\left (1-h_{out}\sum_{i=0}^k \nu(i)\right )}\\
= \inf_g \frac{\sum_{k=1}^{T}g(k)^2\left ( \nu(k)+\nu(k-1) \right )}{\sum_{k=1}^{T}\left (\sum_{i=1}^k g(i) \right )^{\!\!2}\left (1-h_{out}\sum_{i=0}^k \nu(i)\right )},\end{align*}
taken over all functions $g:\N\to\R$.

To simplify, we write the volume growth and decay bounds as $\mu(k) = 1-h_{out}\sum_{i=0}^k \nu(i)$ and $\zeta(k) = \nu(k)+\nu(k-1)$ if $1\leq k\leq T$, and $\mu(k) = \zeta(k) = 0$ if $k\geq T$.  We have \begin{align*}2R = \inf_g \frac{\sum_{k=1}^{T}g(k)^2\zeta(k)}{\sum_{k=1}^{T}\left (\sum_{i=1}^k g(i) \right )^{\!\! 2}\mu(k)}.\end{align*}
The result follows from Proposition 1 in \cite{Miclo}.
\end{proof}

An immediate corollary is a bound on the spectral gap, obtained by combining Theorem \ref{thm:vgB} with the bound $\lambda_2\leq R$.
\begin{theorem}\label{thm:lambda_2} The inequality 
	$$\lambda_2(G) \leq \frac{1}{2B}$$ 
holds, where $$B = \sup_{n\geq 1}\left (\sum_{k=n}^T \left (1-h_{out}\sum_{i=0}^k \nu(i)\right )\right )\left (\sum_{k=1}^n\frac{1}{\nu(k)+\nu(k-1)}\right ).$$ \end{theorem}

A case of particular interest is when $\Sigma = \max_{i\in \Z}|d^{-1}_{\Sigma}(i)|$.  In this case we may set $\nu\equiv 1$.  

\begin{corollary}\label{corr:lambda2} If the vertex-isoperimetric cut-set $\Sigma$ satisfies $\displaystyle \Sigma = \max_{i\in \Z}|d^{-1}_{\Sigma}(i)|$, then $$\lambda_2\leq \frac{27}{2}h_{out}^2(1+o(1)).$$\end{corollary}

The proof is found in Example \ref{ex:constant}. Under these hypotheses the Cheeger lower bound $\lambda_2 \geq c*h_{out}^2/d$ is tight up to a linear factor of $d$.

Observe that this is a related result to Theorem \ref{thm:biglevel}.  WLOG assume $h_{out} = |\Sigma|/|V^+|$, this behaves similarly to the term $\alpha = |\Sigma|/|V|$ in that theorem.

\subsection{Results for the higher Cheeger constants}
We define the {\bf higher order, outer vertex Cheeger constant} to be 
$$h_{out}(n)=\min_{V_1, \ldots , V_n} \max_i \left \{ \frac{|\partial_{out} V_i|}{|V_i|} \right \},$$ 
where $V_1, V_2, \ldots, V_n\subset V$ are non-empty, pairwise disjoint, and have the property that $\cup_{i=1}^n V_i =V.$ 
Our main focus in this subsection is to develop enough of the properties of $h_{out}(n)$ to give the following analogue of Corollary \ref{corr:lambda2} for the higher eigenvalues:
\begin{theorem}\label{thm:higherBuser} Assume that $\nu(i)=1$ for all $i \in [T^-, T^+].$ If $n\geq 2$ and $h_{out}(n)<1,$ then we have 
\begin{align*}\lambda_k(G)\leq  k^2h_{out}(n)^2 \parens{\frac{27\pi^2}{16}+o(1)}.
\end{align*}
\end{theorem}

The proof of Theorem \ref{thm:higherBuser} is found in Example \ref{ex:HigherEigen} and the remaining portion of this section is devoted to developing the properties of $h_{out}(n)$ enough to support the proof of this result.

The concept of the higher Cheeger constant of graphs, as well as the first Cheeger-type and Buser-type inequalities for the higher Cheeger constants (in various forms) have been studied by many authors; see for instance \cite{LOGT,LRT12, Miclo2}. 
We will assume that 
    	$$h_{out}(n)=\max_{i=1,2,\ldots, n}
        	\left \{ \frac{|\partial_{out}V_i|}{|V_i|} \right \} = \frac{|\partial_{out}V_n|}{|V_n|}.$$
	For convenience and without loss of generality, we assume that 
    	$$\frac{|\partial_{out} V_1|}{|V_1|} \leq \frac{|\partial_{out}V_2|}{|V_2|} \leq \cdots \leq \frac{|\partial_{out}V_n|}{|V_n|}.$$
    Further, we may also construct the $V_i$ such that if 
    	$$\frac{|\partial_{out}V_{k-1}|}{|V_{k-1}|}=\frac{|\partial_{out}V_k|}{|V_k|},$$
then $|V_{k-1}| \geq |V_k|.$

To prove bounds on $\lambda_n(G)$ with respect to $h_{out}(n),$ there are two plausible approaches:
	\begin{enumerate}
		\item Prove a monotonicity-type estimate bounding $h_{out}(n)$ from below by $h_{out}(2).$ Then apply these estimates directly to Lemma \ref{lem:d_lowerbound}.
        \item Prove an analog to Lemma \ref{lem:d_lowerbound} for $h_{out}(n)$ in place of $h_{out}(2).$
	\end{enumerate}    
While we take approach 1 for convenience, we mention approach 2, since we would be interested in any work in this direction that might produce better bounds.  The fact that $h_{out}(n)\geq h_{out}(2)$ follows immediately from the following result.

	\begin{proposition}\label{prop:hmonotonicity2} With $h_{out}(n)$ defined as above, for $n \geq 3,$ we have 
	$$h_{out}(n-1) \leq h_{out}(n).$$
	\end{proposition}
\begin{proof}
	Using the notation established in this section, we remind that reader that 
    \begin{equation*}
    	h_{out}(n)=
        \max_{1\leq i \leq n}
        \frac{|\partial_{out}V_i|}{|V_i|}
        =\frac{|\partial_{out}V_n|}{|V_n|}.
    \end{equation*}
Consider the sets $V_1,\dots, V_n$ that optimize $h_{out}(n)$.  We form a collection of $n-1$ sets that will be a candidate to optimize $h_{out}(n-1)$ by merging $V_1$ and $V_2$ to make $V^* = V_1\cup V_2$ and by retaining the other $n-2$ sets.

Observe that $$\frac{|\partial_{out}V^*|}{|V^*|}\leq \frac{|\partial_{out}V_1| + |\partial_{out}V_2|}{|V_1|+|V_2|}\leq \max\left \{\frac{|\partial_{out}V_1|}{|V_1|}, \frac{|\partial_{out}V_2|}{|V_2|}\right \},$$
where the first inequality relies on the fact that $|\partial_{out}(V_1\cup V_2)|\leq |\partial_{out}V_1| + |\partial_{out}V_2|$.   The second inequality uses the rule $\frac{a+b}{c+d}\leq \max\{\frac{a}{c},\frac{b}{d}\}$  when $a,b,c,d > 0$.

Combining this bound with the monotonicity of $\frac{|\partial_{out}V_i|}{|V_i|}$, we find that $$\frac{|\partial_{out}V^*|}{|V^*|}\leq \max_{1\leq i \leq n}
\frac{|\partial_{out}V_i|}{|V_i|}
=\frac{|\partial_{out}V_n|}{|V_n|};$$ that is, the maximum ratio on these $n-1$ sets that partition $V$ is $\frac{|\partial_{out}V_n|}{|V_n|}$.  Because $h_{out}(n-1)$ is the minimum value of the maximum ratio taken over any choice of $n-1$ sets that partition $V$, we find that $$h_{out}(n-1)\leq \frac{|\partial_{out}V_n|}{|V_n|} = h_{out}(n).$$

\end{proof}

\begin{remark}\label{rmk:d_lowerbound} Recall that the terms $h_{out}(n)$ and $\sum_{i=0}^k \nu(i)$ are both positive. Using these facts, the following bound is immediate from combining Proposition \ref{prop:hmonotonicity2} with Lemma \ref{lem:d_lowerbound}. So, with the same notation and assumptions as in Lemma \ref{lem:d_lowerbound}, we have 
	$$|\Sigma_k| \geq |\Sigma| \left (
    	1-h_{out}(n)
        \sum_{i=0}^k \nu(i)
        \right ).$$
\end{remark}

In the next section, we will cite this remark in the analysis of some examples.

\section{Examples of spectral gap bounds using volume growth}

In this section, we use Theorem \ref{thm:lambda_2} to bound the second eigenvalue by the volume growth.  First we obtain several general bounds depending only on the growth function $\nu(k)$.  Second, we use these results to bound $\lambda_2$ for specific graphs where the growth function is known. In each example where a bound on $\lambda_2(G)$ is computed, we compute $B$ from the statement of Theorem \ref{thm:lambda_2}.

\subsection{Examples of volume growth functions}

\begin{example}\label{ex:exponential}
If $\nu(i)$ is exponential, i.e., $\nu(i) = c^i$ for some value $c>1$, then $T$ satisfies \begin{align*}\frac{c^{T+1}-1}{c-1}\leq \frac{1}{h_{out}}\leq \frac{c^{T+2}-1}{c-1}.\end{align*}
 As such, we have 
 \begin{align*}h_{out}\leq \frac{c-1}{c^{T+1}-1}.\end{align*}
 
Note that it is trivial that $\nu(i) = d\cdot(d-1)^{i-1} < d^i$ is a volume growth bound for all $d$-regular graphs.  This bound is achieved by a tree where $|\Sigma|$ is a single vertex.  So we only need to consider the case $c\leq d$.

If $T\geq n\geq 1$, it follows that  
\begin{align*}\sum_{k=n}^T \left (1-h_{out}\sum_{i=0}^k \nu(i)\right ) &= (T-n+1)-\sum_{k=n}^T h_{out} \frac{c^{k+1}-1}{c-1} \\ 
&\geq (T-n+1) - \sum_{k=n}^T\frac{c^{k+1}-1}{c^{T+1}+1}\\
&=(T-n+1)\left (1+\frac{1}{c^{T+1}-1}\right )-\frac{c^{T+2}-c^{n+1}}{(c-1)(c^{T+1}-1)}\,.\end{align*}
We also have that 
\begin{align*}\sum_{k=1}^n\frac{1}{\nu(k)+\nu(k-1)} = \sum_{k=1}^n \frac{1}{(c+1)c^{k-1}} =\frac{1-c^{-n}}{c-\tfrac{1}{c}} = \frac{c-c^{1-n}}{c^2-1}\,.\end{align*}
Combining the previous two equations, we have 
\begin{align*}B\geq \sup_{T\geq n\geq 1} \left ((T-n+1)\left (1+\frac{1}{c^{T+1}-1}\right )-\frac{c^{T+2}-c^{n+1}}{(c-1)(c^{T+1}-1)}\right )\left (\frac{c-c^{1-n}}{c^2-1}\right ).
\end{align*}

Taking $n = 1$, we find \begin{align*}B&\geq \biggl(T+\frac{T}{c^{T+1}-1}-\frac{c^{T+2}-c^{2}}{(c-1)(c^{T+1}-1)}\biggr)\biggl(\frac{c-1}{c^2-1}\biggr) \\ &\geq\biggl(T+\frac{T}{c^{T+1}-1}-\frac{c}{c-1}\biggr)\biggl(\frac{1}{c+1}\biggr).\end{align*}

On the other hand, for any value $n$ satisfying $1\leq n\leq T$, we have that 

\begin{align*}&\sum_{k=n}^T \left (1-h_{out}\sum_{i=0}^k \nu(i)\right ) \leq T\,, \end{align*} 
and, as a result, it follows that 
\begin{align*}
\sum_{k=1}^n\frac{1}{\nu(k)+\nu(k-1)} = \sum_{k=1}^n \frac{1}{(c+1)c^{k-1}} =
\frac{1-c^{-n}}{c-\tfrac{1}{c}}\leq \frac{1}{c-\tfrac{1}{c}} = \frac{c}{c^2-1}\,.
\end{align*} 
So, combining all parts, we see that \begin{align*}\biggl(T+\frac{T}{c^{T+1}-1}-\frac{c}{c-1}\biggr)\biggl(\frac{1}{c+1}\biggr)\leq B\leq T\frac{c}{c^2-1}.\end{align*}

In particular, if $c\geq 1+\ep,$ for a fixed $\ep>0$, then $B = \Theta(T/c)$ and $\lambda_2=O(c/T).$
\end{example}

\begin{example}\label{ex:exponential2} Of particular interest is the case that $\nu(0) = 1$, $\nu(i) = d c^{i-1}$ if $i\geq 1$, where $d$ is the common degree of vertices in the graph and $c>1$.  This is the form of Theorem \ref{thm:vg}. Proceeding in the same way as the previous example, we see that $T$ satisfies 
\begin{align*}1+d\frac{c^{T}-1}{c-1}\leq \frac{1}{h_{out}}\leq 1 + d\frac{c^{T+1}-1}{c-1}.\end{align*}
It follows that 
\begin{align*}h_{out}\leq \frac{c-1}{c-1+d(c^{T}-1)}.\end{align*}
In the case where $T\geq n\geq 1$, we have  
\begin{align*}
\sum_{k=n}^T \left (1-h_{out}\sum_{i=0}^k \nu(i)\right ) &= (T-n+1)-\sum_{k=n}^T h_{out}  \frac{c-1+d(c^{k}-1)}{c-1} \\ 
 & \geq (T-n+1) - \sum_{k=n}^T\frac{c-1+ d(c^{k}-1)}{c-1+d(c^{T}-1)}\\ 
 &=(T-n+1)\left (1+\frac{d+1-c}{c-1+d(c^{T}-1)}\right )\\
 &\hspace{.2in} -\frac{d(c^{T+1}-c^{n})}{(c-1)(c-1+d(c^{T}-1))}.\end{align*} 
 In addition, we find that 
\begin{align*}
\sum_{k=1}^n\frac{1}{\nu(k)+\nu(k-1)} &= \frac{1}{1+d}+\sum_{k=2}^n \frac{1}{d(c+1)c^{k-2}} \\
 &=\frac{1}{1+d}+\frac{1}{d}\cdot\frac{1-c^{1-n}}{c-\tfrac{1}{c}} \\
 &= \frac{1}{1+d}+\frac{1}{d}\cdot\frac{c-c^{2-n}}{c^2-1}.\end{align*}

Combining the previous two equations, we have 
\begin{align*}
B\geq \sup_{T\geq n\geq 1} \left ((T-n+1)\left (1+\frac{d+1-c}{c-1+d(c^{T}-1)}\right )
-\frac{d(c^{T+1}-c^{n})}{(c-1)(c-1+d(c^{T}-1))}\right )\\ 
\cdot \left (\frac{1}{1+d}+\frac{1}{d}\cdot\frac{c-c^{2-n}}{c^2-1}\right ).
\end{align*}

Taking $n = 1,$ we find 
\begin{align*}B &
\geq \left (T\left (1+\frac{d+1-c}{c-1+d(c^{T}-1)}\right )-
\frac{d(c^{T+1}-c)}{(c-1)(c-1+d(c^{T}-1))}\right )\\ 
&\hspace{.2in} \cdot \left (\frac{1}{1+d}+\frac{1}{d}\cdot\frac{c-c}{c^2-1}\right ) \\ 
&\geq \left (T+\frac{T(d+1-c)}{c-1+d(c^{T}-1)}-\frac{c}{c-1}\right )\left(\frac{1}{1+d}\right ).
\end{align*}

On the other hand, if $1\leq n\leq T$, we have 
\begin{align*}
\sum_{k=n}^T \left (1-h_{out}\sum_{i=0}^k \nu(i)\right ) \leq T\,, 
\end{align*}
and

\begin{align*}\sum_{k=1}^n\frac{1}{\nu(k)+\nu(k-1)} = \frac{1}{1+d}+\frac{1}{d}\cdot\frac{c-c^{2-n}}{c^2-1}\leq \frac{1}{1+d}+\frac{1}{d}\cdot\frac{c}{c^2-1}\,.\end{align*} Thus, combining all parts, we see that \begin{align*}\biggl(T+\frac{T(d+1-c)}{c-1+d(c^{T}-1)}-\frac{c}{c-1}\biggr)\biggl(\frac{1}{1+d}\biggr)\leq B\leq T\biggl(\frac{1}{1+d}+\frac{1}{d}\cdot\frac{c}{c^2-1} \biggr).\end{align*}

If $c\geq 1+\ep$ for a fixed value $\ep > 0$, then $B = \Theta(T/d)$ and $\lambda_2=O(d/T)$.
\end{example}

\begin{example}\label{ex:constant}\item If $\nu(i) = 1$ for all $i\geq 0$, then $T$ satisfies $T+1\leq \frac{1}{h_{out}}\leq T+2.$

\begin{align*}
B &= \sup_{n\geq 1}\left (\sum_{k=n}^T\big (1-h_{out}(k+1)\big )\right )
	\left (\sum_{k=1}^n \frac{1}{2}\right ) \\ 
    &= \sup_{n\geq 1}\left ((T-n+1)-h_{out}\left [\binom{T+1}{2}-\binom{n}{2}\right ]\right )
    \left(\frac{n}{2}\right ) \\ 
    &\geq \sup_{n\geq 1}\biggl(T+1-n-\frac{1}{T+1}\frac{T^2+T-n^2+n}{2}\biggr)\cdot\frac{n}{2} \\ &= \frac{T^2}{27}(1\pm o(1)) \\
    &= \frac{1}{27h_{out}^2}(1\pm o(1)).
\end{align*}

Here the supremum for $B$ is achieved when $n$ is roughly equal to  $T/3$.  
It follows that $\lambda_2\leq \frac{27}{2}h_{out}^2(1+o(1))$.  In this case the Cheeger lower bound $\lambda_2 \geq c*h_{out}^2/d$ is tight up to a linear factor of $d$.
Note that this is a case of Theorem \ref{thm:biglevel} in which the Cheeger cut-set is also the largest set.  This follows because for all $i$, $|\Sigma_i|\leq \nu(i)|\Sigma_0| = |\Sigma_0|$, where $|\Sigma_0|$ is by assumption the Cheeger cut-set.
\end{example}

\begin{example}
\label{ex:polynomial}
\item If $b\geq 1$ is a constant so that $\nu(i) = 1+i^b$ for all $i\geq 0,$ then $T$ satisfies 
\begin{align*}\frac{1}{h_{out}} \geq \sum_{i=0}^T (1+i^b)\geq \int_0^T x^b\, dx=\frac{T^{b+1}}{b+1}.
\end{align*}
For the computation of $B,$ we use the inequality 
    \begin{equation}\label{eq:Referee}
        \sum_{i=0}^k \left (1+i^b\right ) \leq 3\cdot \sum_{i=1}^k i^b \leq 3\cdot \frac{(k+1)^{b+1}-1}{b+1}.
    \end{equation}
The first inequality follows from 
    $$\sum_{i=1}^k i^b \geq \sum_{i=1}^k i=\frac{k(k+1)}{2}\geq \frac{k+1}{2}.$$ So then $2\sum_{i=1}^k i^b \geq k+1$ and so 
    $$\sum_{i=0}^k i^b \leq (k+1)+\sum_{i=1}^k i^b \leq 3 \cdot \sum_{i=1}^k i^b.$$
The second inequality in (\ref{eq:Referee}) follows from 
    $$\sum_{i=1}^k i^b\leq \int_1^{k} x^b \, dx=\frac{k^{b+1}-1}{b+1},$$
since $b\geq 1.$

Thus, we have 
\begin{align*}B &= 
\sup_{n\geq 1}\left (\sum_{k=n}^T\left (1-h_{out}\left (\sum_{i=0}^k(1+i^b)\right )\right )\right )
\left (\sum_{k=1}^n \frac{1}{2+k^b+(k-1)^b}\right )\\
&\geq \sup_{n\geq 1}\left (\sum_{k=n}^T\left (1-3h_{out}\left (\frac{k^{b+1}-1}{b+1}\right )\right )\right )
\Theta(1)
\\
&\geq \sup_{n\geq 1}\left((T-n+1)-3h_{out}\left(\sum_{k=n}^T\frac{k^{b+1}-1}{b+1}\right)\right )
\Theta(1)\\
&= \sup_{n\geq 1}\Theta\left ((T-n+1)-3h_{out}\frac{T^{b+2}-n^{b+2}}{(b+1)(b+2)}(1+o(1))\right )\\
&\geq \sup_{n\geq 1}\Theta\left ((T-n+1)-3\cdot \frac{T^{b+2}-n^{b+2}}{T^{b+1}(b+2)}(1+o(1))\right )\\
&= \Theta(T)
\end{align*}

So we conclude that $\lambda_2= O(1/T) = O(h_{out}^{1/b+1})$.

\smallskip
This example represents {\it polynomial} volume growth.  Recall that in the setting of Ollivier curvature, every graph with positive curvature has polynomial volume growth with some positive integer $b$.  But the Buser bound we hoped to achieve is $\lambda_2 = O(h_{out}^2)$.  The reason for the difference may be that Paeng's polynomial volume growth bound is a correct bound for the volume growth around any initial set.  In this section we are only concerned with bounding volume growth around the Cheeger-achieving cut-set.  For that set, a tighter bound may apply.  Our next examples are instances of this phenomenon, where the volume growth is much slower around the Cheeger cut-set than around general vertex sets.

\end{example}

We will now provide an application of Theorem \ref{theo:eigenhigher} to Buser-type inequalities for combinations of higher eigenvalues and the higher Cheeger constants.
\begin{example}\label{ex:HigherEigen}
Assume that $\nu(i)=1$ for all $i \in [T^-, T^+]$ and $h_{out}(n)<1$ for some $n \geq 2.$ Due to the symmetry of this example, we abuse notation slightly to simplify the presentation, defining $B^{\pm}$ to be the $T^{\pm} \times T^{\pm}$ Toeplitz, tridiagonal matrix defined by 
	\begin{equation*}
    	B^{\pm}_{ij}=\begin{cases} 4, & \text{ if }i=j\\
        -2, &\text{ if }|i-j|=1\\ 
        0, &\text{ otherwise}.\end{cases}
    \end{equation*}
Because $B^{\pm}$ differs from $A^{\pm}$ in only the $(T^{\pm},T^{\pm})$ entry, we have that 
	\begin{equation*}
    	\langle g,B^{\pm} g\rangle_{\pm}-\langle g, A^{\pm}g\rangle_{\pm}=
        B^{\pm}_{T^{\pm},T^{\pm}} g(T^{\pm})^2-A^{\pm}_{T^{\pm},T^{\pm}}g(T^{\pm})^2=2g(T^{\pm})^2 \geq 0.
    \end{equation*}
Note that the eigenvalues of the matrix $B^{\pm},$ denote them $\psi_k,$ are given in closed form by
	\begin{equation}\label{eq:ToeEigen}
    	\psi_k=4\left (1-\cos\left ( \frac{k\pi}{T^+-T^-}\right ) \right ),
    \end{equation}
see, for instance, Theorem~2.2 in \cite{KST99}, wherein a new approach was proposed (with extensions to Toeplitz-like matrices), while \cite{Smith85} details the classical treatment. 

Now we combine Equation \ref{eq:ToeEigen} with Theorem \ref{theo:eigenhigher} which implies that 
\begin{equation}\label{eq:higherBound}
	\lambda_k(G)\leq 
    2\cdot\min_{\left \lceil \frac{k}{2}\right\rceil 
    \leq t \leq \min\{T^+,T^-\}}
    \frac{1-\cos \left (\frac{
    \left \lceil \frac{k}{2}\right \rceil}{t+1}\pi \right )}
    {1-h_{out}(n)(t+1)},
\end{equation}
where the denominator follows from Remark \ref{rmk:d_lowerbound}. In particular, the weight $\mu(k)$ from Theorem \ref{theo:eigenhigher} is given by 
	\begin{equation*}
    	\mu(k)=1-h_{out}(n)\sum_{i=0}^k \nu(i)
        =1-h_{out}(n)(k+1).
    \end{equation*}

It remains to minimize the right hand side of Equation \ref{eq:higherBound}. 
We will use the simple bound that if $0 \leq x \leq \pi$, with $\frac{2}{\pi^2}x^2 \leq 1-\cos(x) \leq \frac{1}{2}x^2.$

From Equation \ref{eq:higherBound}, we obtain 
\begin{equation}
\lambda_k(G)\leq 
    2\cdot\min_{\left \lceil \frac{k}{2}\right\rceil 
    \leq t \leq \min\{T^+,T^-\}}
    \frac{
    \left \lceil \frac{k}{2}\right \rceil^2\frac{\pi^2}{2}}
    {[1-h_{out}(n)(t+1)](t+1)^2}.
\end{equation}

Observe that in this step of our estimate, we use a bound that is tight up to a constant factor $\pi^2/4.$ One might be tempted to use a better approximation for $\cos(x)$, but this factor gives an upper bound on the potential improvement from that method.

Elementary calculus reveals that the minimum is achieved when $(t+1) = \frac{2}{3h_{out}(n)}.$ Of course this may be not an integer: we will set
\begin{align*}
t+1 = \left \lceil \frac{2}{3h_{out}(n)}\right  \rceil \text{ if }\frac{k}{2}\leq\frac{2}{3h_{out}(n)}\leq \min\{T^+,T^-\}.
\end{align*}

In this case, we find that \begin{align*}\lambda_k(G)\leq
    2\cdot \frac{ \left \lceil \frac{k}{2}\right \rceil^2\frac{\pi^2}{2}}
    {[1-h_{out}(n)\lceil \frac{2}{3h_{out}(n)} \rceil ](\lceil \frac{2}{3h_{out}(n)} \rceil)^2} = k^2h_{out}(n)^2 \parens{\frac{27\pi^2}{16}+o(1)}.\end{align*}

For this problem we have $1/h_{out}(n) < 2+\min\{T^+,T^-\}$, and so $\frac{2}{3h_{out}(n)}\leq \min\{T^+,T^-\}$ as long as $\min\{T^+,T^-\}\geq 4$

We will not analyze the case that $\frac{2}{3h_{out}(n)}<k/2$ or that $\min\{T^+,T^-\}<4$.  It is easy to check that both cases give (trivial) bounds of the form $\lambda_k \leq C$ for a universal constant $C.$
\end{example}

\subsection{Examples of specific graphs}
We will now test our methods on several concrete examples. For these examples, information about the spectrum is already known, allowing us to compare the results.

\begin{example}[Hypercube]
The hypercube $\Omega_d$ is commonly expressed as the graph with vertex set $\{0,1\}^d$ and $x\sim y$ if and only if $x$ and $y$ disagree in exactly one coordinate.  With this notation, we define the $k$-slice $A_k\subset V$ to be the set of vertices that are $1$ in exactly $k$ coordinates.  It is clear that $|A_k| = \binom{d}{k}$.

It is known that $h_{out}$ is achieved by the $\lfloor d/2\rfloor$-slice $\Sigma$, with $h_{out} = \Theta(1/\sqrt{d})$ \cite{Har66}.  With this choice of $\Sigma$, we see that $\dist^{-1}(i) = A_{\lfloor d/2\rfloor+i}$, and 
$$|\dist^{-1}(i)| = \binom{d}{\lfloor d/2\rfloor+i}\leq  \binom{d}{\lfloor d/2\rfloor} = |\Sigma|.$$
As such, we may set $\nu(i) = 1$, and we have 
$$T = \left \lfloor\frac{1}{h_{out}}-1\right \rfloor = \Theta(\sqrt{d}).$$

By the results of Example \ref{ex:constant},  $\lambda_2 \leq \tfrac{27}{2}h_{out}^2(1+o(1))$, thus $\lambda_2 = O(1/d)$.  It is well-known that the actual value of $\lambda_2$ is indeed $\Theta(1/d)$.
\end{example}

\begin{example}[Discrete torus]
If $C_n$ is the $n$-cycle for $n\geq 3$, the discrete torus $C_n^d$ is the $2d$-regular graph $C_n\square C_n \square \cdots \square C_n$.  
It is understood that $h_{out}$ is achieved by the ball $B(x,\lceil\tfrac{dn}{4}\rceil-1)$ with $\Sigma = S(x,\lceil\tfrac{dn}{4}\rceil)$, where $x$ is an arbitrary (fixed) vertex \cite{BoL06}.  The level sets are $\dist^{-1}_\Sigma(i) = S(x,i+\lceil\tfrac{dn}{4}\rceil)$ with $|\dist^{-1}_\Sigma(i)|\leq |\Sigma|.$  We will give a brief argument that $\tfrac{2}{nd}(1+o(1) < h_{out} <\tfrac{4}{n}(1-o(1))$.

First note that $|\Sigma|< 2n^{d-1}$, as the latter is achieved by the boundary of the candidate cut-set bounded by two parallel $d-1$-planes seperated by a distance $\lfloor n/2\rfloor$.  It follows that $h_{out} < n^{d-1}/(\tfrac{1}{2}n^d(1+o(1)) = 4/n(1+o(1)$.

Next, consider the set $A\subset C_n^{d-1}$ defined to contain those $a$ for which there is an element of $\Sigma$ whose first $d-1$ entries are $a$.  Let $\overline{x}$ be the first $d-1$ entries of $x$, $A = \bigcup_{k=0}^{\lfloor n/2\rfloor} S(\overline{x},\lceil\tfrac{dn}{4}-k\rceil)$.  Inductively, we know that $A$ contains the disjoint union of the $n/2-O(1)$ largest shells around $\overline{x}$ in $C_n^{d-1}$; as there are $nd/2 + O(1)$ shells in total and we take the largest fraction $1/d - o(1)$ to form $A$, $|A|\geq \parens{\tfrac{1}{d}-o(1)}|C_n^{d-1}| = (n^{d-1}/d)(1-o(1))$.  Clearly $|A| \leq |\Sigma|$, it follows that $h_{out} > (n^{d-1}/d)/(\tfrac{1}{2}n^d)(1-o(1)) = \tfrac{2}{nd}(1-o(1))$.

And so we have determined that $h_{out} = \tfrac{1}{n}$ is tight within a factor linear in $d$.  Proceeding similarly to the hypercube, we may use $\nu(i) = 1$ as in Example \ref{ex:constant} to see that $\lambda_2\leq \tfrac{27}{2}h_{out}^2(1+o(1))$, thus 
$$\lambda_2 \leq \Theta_d(\tfrac{1}{n^2})\,.$$  It is well-known that the actual value is $\lambda_2 = \Theta(\tfrac{1}{n^2})$, so our estimate is tight up to a factor depending on $d$.
\end{example}

\bibliographystyle{plainnat}
\bibliography{bibliography}

\end{document}